\documentclass[a4paper,10pt]{amsart}
\usepackage[arrow,matrix]{xy}
\usepackage{color}
\usepackage{amsmath,amssymb,amscd,bbm,amsthm,mathrsfs,dsfont,enumerate}
\theoremstyle{plain} \textwidth=36pc \textheight=51pc

\newtheorem{theorem}{Theorem}[section]
\newtheorem{lemma}[theorem]{Lemma}
\newtheorem{example}[theorem]{Example}
\newtheorem{proposition}[theorem]{Proposition}
\newtheorem{corollary}[theorem]{Corollary}
\theoremstyle{definition}
\newtheorem{definition}[theorem]{Definition}
\newtheorem{remark}[theorem]{Remark}
\numberwithin{equation}{section}

 \DeclareMathOperator{\Ext}{Ext}
\DeclareMathOperator{\Hom}{Hom} \DeclareMathOperator{\Tor}{Tor}

\DeclareMathOperator{\e}{\epsilon}

\DeclareMathOperator{\ch }{char}

\DeclareMathOperator{\gr }{gr}
\DeclareMathOperator{\de }{d}

\DeclareMathOperator{\Der }{Der}
\DeclareMathOperator{\HP }{HP}
\DeclareMathOperator{\Sym }{Sym}
\DeclareMathOperator{\Asym }{Asym}

\begin{document}
\title{Universal enveloping algebras of Poisson Hopf algebras}
\author{Jiafeng L\"u, Xingting Wang and Guangbin Zhuang}
\address{L\"u: Department of Mathematics, Zhejiang Normal University, Jinhua 321004, China}
\email{jiafenglv@gmail.com}

\address{Wang: Department of Mathematics,
University of Washington, Seattle, Washington 98195, USA}

\email{xingting@uw.edu}

\address{Zhuang: Department of Mathematics,
University of Southern California, Los Angeles 90089-2532, USA}

\email{gzhuang@usc.edu}

\begin{abstract} 
For a Poisson algebra $A$, by exploring its relation with Lie-Rinehart algebras, we prove a Poincar\'e-Birkoff-Witt theorem for its universal enveloping algebra $A^e$. Some general properties of the universal enveloping algebras of Poisson Hopf algebras are studied. Given a Poisson Hopf algebra $B$, we give the necessary and sufficient conditions for a Poisson polynomial algebra $B[x; \alpha, \delta]_p$ to be a Poisson Hopf algebra. We also prove a structure theorem for $B^e$ when $B$ is a pointed Poisson Hopf algebra. Namely, $B^e$ is isomorphic to $B\#_\sigma \mathcal{H}(B)$, the crossed product of $B$ and $\mathcal{H}(B)$, where $\mathcal{H}(B)$ is the quotient Hopf algebra $B^e/B^eB^+$.
\end{abstract}



\maketitle

\medskip
\section{Introduction}

The notion of Poisson Hopf algebras arises naturally in the study of Poisson geometry and quantum groups. For example, the coordinate ring of a Poisson algebraic group is a Poisson Hopf algebra \cite[Definition 3.1.6]{KS}.
Recently, connected Hopf algebras are studied in a series of papers \cite{BOZZ, WZZ, WW, Zh1}. In \cite{Zh1}, the third named author showed that, for a connected Hopf algebra $H$, the associated graded algebra $\gr H$ with respect to the coradical filtration is always commutative. Therefore $\gr H$ carries an induced Poisson structure, which actually makes $\gr H$ into a Poisson Hopf algebra. This observation suggests, in a very rough sense, that one should think of connected Hopf algebras as deformations of Poisson Hopf algebras. While the study of Poisson Hopf algebras is interesting on its own, we certainly hope that it will help us to understand Hopf algebras in general.

For a Poisson algebra $A$, one can define its universal enveloping algebra $A^e$, which basically transfers the Poisson structure of $A$ to the algebra structure of $A^e$. Generally speaking, the algebra $A^e$ could be very complicated and highly non-commutative. For example, in \cite{U}, Umirbaev proved that the universal
enveloping algebras of the Poisson symplectic algebra $P_n$ is isomorphic to $A_n\otimes A_n^{op}$, where $A_n$ is the $n$-th Weyl algebra. If we start with a Poisson Hopf algebra $B$, then $B^e$ is actually a Hopf algebra \cite{Oh2}, whose module category is equivalent to the category of Poisson modules over $B$ as monoidal categories. Hence, from a categorical point of view, the Hopf algebra $B^e$ carries all information we need to understand the Poisson Hopf algebra $B$, which motivates us to study the structure of $B^e$. 

The paper is organized as follows. 

In Section \ref{pre} and Section \ref{example}, we briefly review some basic concepts related to Poisson Hopf algebras and give some examples. Section \ref{poly} is devoted to the study of Poisson polynomial algebras. To be specific, for a given Poisson Hopf algebra $B$, we give the necessary and sufficient conditions for a Poisson polynomial algebra $B[x; \alpha, \delta]_p$ to be a Poisson Hopf algebra. 

In Section \ref{Puniversal}, we review the definition and some basic properties of $A^e$, the enveloping algebra of a Poisson algebra $A$. A well-known result in \cite{H} states that the K\"alher differentials $\Omega_A$ of a Poisson algebra is a Lie-Rinehart algebra. By using the universal property of $A^e$ \cite{Oh2}, we show that $A^e$ is isomorphic to $V(A, \Omega_A)$, the enveloping algebra of the Lie-Rinehart algebra $\Omega_A$. 
Consequently, a Poincar\'e-Birkoff-Witt theorem for $A^e$ is derived.

In the rest of the sections we focus on the structure of $B^e$ where $B$ is a Poisson Hopf algebra. In \cite{Oh2}, it is shown that $B^e$ is Hopf algebra with $B$ as a Hopf subalgebra. By using standard techniques from Hopf algebra theory, we show in Section \ref{Hopf} that 
$B^e$ is always a free module over $B$ and an injective comodule over $\mathcal{H}(B)$, where $\mathcal{H}(B)=B^e/ B^eB^{+}$. Moreover, when $B$ is a pointed Poisson algebra, we are able to prove that 
$$B^e\cong B\#_\sigma \mathcal{H}(B)$$
as algebras, where $B\#_\sigma \mathcal{H}(B)$ is the crossed product of $B$ and $\mathcal{H}(B)$. This result involves showing that the extension $B\subset B^e$ has the normal basis property and is Galois.
In Section \ref{graded}, a more special case, where $B$ is a pointed Poisson Hopf algebra with $\{g, t\}=0$ for any $g, t\in G:=G(B)$, is considered. In this case, $C=\gr B$ inherits a Poisson structure from $B$. Also, the canonical biproduct decomposition \cite{Ra3}
$$C\cong R\#kG$$
is compatible with the Poisson structure in the sense that both $R$ and $kG$ are Poisson subalgebras of $C$. Now the structure of $C^e$ becomes rather transparent. In fact, $R^e$ is a $(kG)^e$-module algebra and one has
$$C^e\cong R^e\#(kG)^e.$$

Throughout this paper, let $k$ denote a base field. All algebras, coalgebras and tensor products are taken over $k$ unless otherwise stated.
\subsection*{Acknowledgments} The authors thank Ken Brown, Gigel Militaru, Susan Montgomery, Ualbai Umirbaev and James Zhang for reading earlier versions of this paper. The third author would also like to thank Pavel Etingof for explaining some arguments in the paper \cite{EK}. Jiafeng L\"u is supported by National Natural Science Foundation of China (No. 11001245, 11271335 and 11101288).

\medskip
\section{Preliminaries}\label{pre}
In this section we will recall some definitions and basic results.

\subsection{Coalgebras and Hopf algebras.} For the definition and basic properties of coalgebras and Hopf algebras, a comprehensive reference is \cite{Mo}. We recall some terminology used in this paper. 

For a coalgebra $C$, we denote the comultiplication and the counit by $\Delta$ and $\e$, respectively. For any $x\in C$, we write
$\Delta(x)=x_1\otimes x_2$, which is basically the Sweedler's notation with the summation sign $\Sigma$ omitted. Let $G(C)$ be the set of the group-like elements in $C$, and $C^+$ be the kernel of the counit. The {\it coradical}  $C_0$ of $C$ is defined to be the sum of all simple subcoalgebras of $C$. The coalgebra $C$ is called {\it pointed} if $C_0=kG(C)$, and {\it connected} if $C_0$ is one-dimensional. Also, we use $\{C_n\}_{n=0}^{\infty}$ to denote the coradical filtration of $C$ \cite[5.2.1]{Mo}. For a pointed Hopf algebra $H$, the coradical filtration
$\{H_n\}_{n=0}^{\infty}$ is a Hopf algebra filtration \cite[p. 62]{Mo} so that the associated graded algebra, which we denote by $\gr H$, is also a Hopf algebra.

In \cite{AC}, Andruskiewitsch-Cuadra introduced the {\it standard filtration} $\{H_{[n]}\}_{n\ge 0}$ for a Hopf algebra $H$. By definition, $H_{[0]}$ is the subalgebra generated by $H_0$ and 
$$H_{[n]}=\wedge^{n+1}H_{[0]},$$
where $\wedge$ denotes the wedge of $k$-spaces \cite[5.2]{Mo}.
The standard filtration possesses a lot of nice properties. For example, if the antipode $S$ of $H$ is injective (which is true when $H$ is commutative), then $\{H_{[n]}\}_{n\ge 0}$ is a Hopf filtration \cite[Lemma 1.1]{AC}.

\subsection{Poisson Hopf algebras}
\begin{definition} A Poisson algebra is a commutative algebra $A$ equipped with a Lie bracket $\{\cdot,\cdot\}$ such that 
$$\{a, bc\}=\{a, b\}c+b\{a, c\},$$
for any $a, b, c\in A$.
\end{definition}

\begin{proposition}\label{tensor}
Let $A$ be a Poisson algebra with a Poisson bracket $\{\cdot ,\cdot \}$. Then $A\otimes A$ becomes a Poisson algebra via the following multiplication and bracket:
\begin{align}
(a\otimes b)\cdot(c\otimes d)&:=ac\otimes bd,\\
\{a\otimes b, c\otimes d\}&:=ac\otimes \{b, d\}+\{a, c\}\otimes db.
\end{align}
\end{proposition}

A typical way of constructing Poisson algebras is as follows. Let $R$ be a filtered algebra such that the associated graded algebra $A$ is commutative. Then the filtration on $R$ induces a Poisson structure on $A$. In fact, let $\overline{a}\in A_n$ and $\overline{b}\in A_m$ where $a\in R_n$ and $b\in R_m$. Since $A$ is commutative, $ab-ba\in R_{n+m-1}$, which gives  an element $\overline{ab-ba}\in A_{n+m-1}$. Now define $\{\overline{a}, \overline{b}\}=\overline{ab-ba}$. It is easy to check that $\{,\}$ is a Poisson bracket on $A$.

\begin{example}\label{envelop}
{\rm Let $R=U(\mathfrak{g})$, where $\mathfrak{g}$ is a finite-dimensional Lie algebra with Lie bracket $[\,,\,]$. Then $R$ has a canonical filtration such that the associated graded algebra is $S(\mathfrak{g})$. As a consequence, $S(\mathfrak{g})$ becomes a Poisson algebra. In fact, the Poisson structure on $S(\mathfrak{g})$ recovers the Lie bracket in the sense that 
$$\{a, b\}=[a, b],$$
for any $a, b\in \mathfrak{g}$. We call this Poisson algebra the {\it Poisson symmetric algebra of $\mathfrak{g}$} and denote it by $PS(\mathfrak{g})$.}
\end{example}

\begin{definition}
Let $B$ be a commutative bialgebra with a Poisson bracket $\{,\}$. Then $B$ is called a Poisson bialgebra if
$$\Delta(\{a, b\})=\{\Delta(a), \Delta(b)\},$$
for any $a, b\in B$.
If, in addition, $B$ is a Hopf algebra, then $B$ is called a Poisson Hopf algebra.
\end{definition}

Let $B$ a Poisson bialgebra, then $\e(\{a, b\})=0$ for any $a, b\in B$, where $\e$ is the counit. In fact, the axiom of counit implies that $\e(x)=\e(x_1)\e(x_2)$ for any $x\in B$. Therefore,
\begin{align*}
\e(\{a, b\})&=\e(\{a_1, b_1\})\e(a_2)\e(b_2)+\e(a_1)\e(b_1)\e(\{a_2, b_2\})\\
&= \e(\{a_1\e(a_2), b_1\e(b_2)\})+\e(\{\e(a_1)a_2, \e(a_2)b_2\})\\
&= \e(\{a, b\})+\e(\{a, b\}).
\end{align*}
Consequently, $\e(\{a, b\})=0$ for any $a, b\in B$. Also, it is easy to show that for a Poisson Hopf algebra $B$, the antipode $S$ is a Poisson algebra anti-morphism, that is, $S(\{a, b\})=\{S(b), S(a)\}$ for any $a, b\in B$.

For a Poisson Hopf algebra $B$, we say that it is a {\it pointed} (resp. {\it connected}) Poisson Hopf algebra if it is pointed (resp. connected) as a coalgebra.

\subsection{Poisson modules}
\begin{definition}\label{Pmodules}
Let $A$ be a Poisson algebra. Then a left Poisson module structure on a left $A$-module over a Poisson algebra $A$ is a linear map 
$$\{\cdot, \cdot\}_M: A\otimes M\rightarrow M$$
such that 
\begin{enumerate}
\item $\{ \{a, b\}_A, m\}_M=\{a, \{b, m\}_M\}_M-\{b, \{a, m\}_M\}_M,$
\item $\{ab, m\}_M=a\{b, m\}_M+b\{a, m\}_M$,
\item $\{a, bm\}_M=\{a, b\}\cdot m+b\{a, m\}_M$,
\end{enumerate}
for any $a, b\in A$ and $m\in M$.
\end{definition}
Similarly, one can define right Poisson modules over $A$. From now on, by saying a Poisson module, we mean a left Poisson module, unless otherwise stated.
Let $M$ and $N$ be two Poisson modules over $A$. Then a Poisson module morphism $f: M\rightarrow N$ is an $A$-module map $f$ such that $f(\{a, m\}_M)=\{a, f(m)\}_N$. Hence, for a Poisson algebra $A$, there is the category of Poisson modules over $A$, with Poisson modules and Poisson module morphisms as objects and morphisms, respectively.

\begin{lemma}
Let $B$ be a Poisson bialgebra and let $M, N$ be two Poisson modules over $B$. Then $M\otimes N$ becomes a Poisson module over $B$ via the following action and bracket,
\begin{align*}
a\cdot(m\otimes n)&:= a_1m\otimes a_2n\\
\{a, m\otimes n\} &:=a_1m\otimes \{a_2, n\}+ \{a_1, m\}\otimes a_2n,
\end{align*}
for any $a\in B$, $m\in M$ and $n\in N$.
\end{lemma}
The following proposition is clear.
\begin{proposition}
Let $B$ be a Poisson bialgebra. Then the category of Poisson modules over $B$ is a monoidal category.
\end{proposition}

Even though in our paper we only consider the canonical case where the Poisson algebras are assumed to be commutative, there has been a great effort to study non-commutative Poisson algebras and their modules. For example, Agore and Militaru introduced the notion of Poisson bimodules for non-commutative Poisson algebras in the paper \cite{AM}.

\medskip
\section{Examples of Poisson Hopf algebras}\label{example}
In this section, we give some examples of Poisson Hopf algebras. It is well known that the enveloping algebra $U(\mathfrak{g})$ of a Lie algebra is a Hopf algebra with $\Delta(a)=1\otimes a + a\otimes 1$ for any $a\in \mathfrak{g}$. A similar result holds for the Poisson symmetric algebra $PS(\mathfrak{g})$, as shown in the following example.
\begin{example}
{\rm The Poisson symmetric algebra $PS(\mathfrak{g})$ defined in Example \ref{envelop} is a Poisson Hopf algebra. The comultiplication and antipode are given by 
$$\Delta(a)=1\otimes a +a\otimes 1, \,\,\,\, S(a)=-a,$$
for any $a\in \mathfrak{g}$.  }
\end{example}

In \cite{WZZ, Zh1}, connected Hopf algebras of Gelfand-Kirillov dimension $3$ and $4$ over an algebraically closed field of characteristic zero are classified. It turns out that for any class of such connected Hopf algebras one could construct a class of Poisson Hopf algebras. Unfortunately, the constructions appear to be ad hoc and we have no theoretical explanation. For example, one should compare Example \ref{GK3} with \cite[Example 7.1]{Zh1}, and Example \ref{GK4} with \cite[Example 4.4]{WZZ}.
\begin{example}\label{GK3}
{\rm Let $B$ be the algebra $k[x_1, x_2, x_3]$ with the following Poisson structure
\begin{align*}
\{x_1, x_2\}=0,\\
\{x_3, x_1\}= \lambda_1x_1+ \alpha x_2,\\
\{x_3, x_2\}=\lambda_2x_2,
\end{align*}
where $\alpha=0$ if $\lambda_1\neq \lambda_2$ and $\alpha=0$ or $1$ if $\lambda_1=\lambda_2$. Then $B$ becomes a Poisson Hopf algebra via
\begin{align*}
\Delta(x_1)=1\otimes x_1+x_1\otimes 1,\\
\Delta(x_2)=1\otimes x_2+ x_2\otimes 1,\\
\Delta(x_3)=1\otimes x_3+x_1\otimes x_2-x_2\otimes x_1+x_3\otimes 1.
\end{align*}}
\end{example}

\begin{example}\label{GK4}
{\rm Let $B$ be the algebra $k[X, Y, Z, W]$ with the following Poisson structure
\begin{align*}
\{Y, X\}&=\{Z, X\}=\{Z, Y\}=0,\\
\{W, X\}&=a_{11}X+a_{12}Y,\\
\{W, Y\}&=a_{21}X+a_{22}Y,\\
\{W, Z\}&=(a_{11}+a_{22})Z+\xi_1X+\xi_2Y,
\end{align*}
where $a_{ij}, \xi_i\in k$. Then $B$ becomes a Poisson Hopf algebra via
\begin{align*}
\Delta(X)&=1\otimes X+ X\otimes 1,\\
\Delta(Y)&=1\otimes Y+ Y\otimes 1,\\
 \Delta(Z)&=1\otimes Z+X\otimes Y-Y\otimes X+ Z\otimes 1,\\
\Delta(W)&=1\otimes W+W\otimes 1\\
&\quad+\theta_1(Z\otimes X-X\otimes Z+X\otimes XY+XY\otimes
X)\\
&\quad+\theta_2(Y\otimes Z-Z\otimes Y+XY\otimes Y+Y\otimes XY),
\end{align*}
where $\theta_i\in k$ and at least one of them is non-zero. }
\end{example}

In the previous examples, the Poisson Hopf algebras are all connected. In the following four examples, the Poisson Hopf algebras are pointed but not connected.
\begin{example}\label{typeA}
{\rm Let $B$ be the algebra $k[g^{\pm 1}, x]$. Then there is a unique Hopf algebra structure on $B$ such that $\Delta(g)=g\otimes g$ and $\Delta(x)=x\otimes 1 +g\otimes x$. Moreover, $B$ is a Poisson Hopf algebra via
\begin{align*}
\{x, g\}= \lambda gx,
\end{align*}
where $\lambda\in k$.}
\end{example}
\begin{example}
{\rm Let $B=k[x, y, z, g^{\pm 1}]$. Then $B$ becomes a Hopf algebra via
\begin{align*}
\Delta(g)=g\otimes g,\\
\Delta(x)=x\otimes g+g^{-1}\otimes x,\\
\Delta(y)=y\otimes g^{-1}+g\otimes x,\\
\Delta(z)=z\otimes1+ 1\otimes z.
\end{align*}
The Poisson structure is given by
\begin{align*}
\{x, g\}=\{y, g\}=\{z, g\}=0,\\
\{z, x\}=\lambda x, \{z, y\}=-\lambda y, \\
\{x, y\}=z.
\end{align*}}
\end{example}

\begin{example}\label{group}
{\rm
Let $B=k[x, a^{\pm1}, b^{\pm 1}]$. Then $B$ becomes a Hopf algebra via
\begin{align*}
\Delta(a)=a\otimes a,\\
\Delta(b)=b\otimes b, \\
\Delta(x)=x\otimes ab+ab\otimes x.
\end{align*}
The Poisson structure is given by
\begin{align*}
\{a, b\}=x,\\
\{x, a\}=\lambda ax^2, \{x, b\}=-\lambda bx^2. \\
\end{align*}}
\end{example}
The next example can be viewed as a Poisson version of $U_q(sl_2)$, the quantized enveloping algebra of $sl_2$.
\begin{example}\label{PoissonU}
{\rm
Let $B$ be the algebra $k[E, F, K^{\pm1}]$ with the following Poisson bracket
\begin{align*}
\{E, K\}=\lambda KE,\\
\{F, K\}= -\lambda KF,\\
\{E, F\}=\alpha (K-K^{-1}).
\end{align*}
Then $B$ becomes a Poisson Hopf algebra via
\begin{align*}
\Delta(K)=K\otimes K,\\
\Delta(E)=E\otimes K+1\otimes E,\\
\Delta(F)=F\otimes 1+K^{-1}\otimes F.
\end{align*}}
\end{example}

The following example can be viewed as a Poisson version of $\mathcal{O}_q(SL_2)$.
\begin{example}
{\rm Let $B$ be the Hopf algebra $\mathcal{O}(SL_2)$. As an algebra, $B=k[x_{11}, x_{12}, x_{21}, x_{22}]/\langle x_{11}x_{22}-x_{12}x_{21}-1\rangle$. The coalgebra structure is given by 
$$\Delta(x_{ij})=\sum_sx_{is}\otimes x_{sj}.$$
The Poisson structure is given by 
\begin{align*}
\{x_{11}, x_{12}\}&=\lambda x_{11}x_{12}, & \{x_{11}, x_{21}\}&=\lambda x_{11}x_{21}, &\{x_{11}, x_{22}\}&=2\lambda x_{12}x_{21},\\
\{x_{12}, x_{21}\}&=0, & \{x_{12}, x_{22}\}&=\lambda x_{12}x_{22}, &\{x_{21}, x_{22}\}&=\lambda x_{21}x_{22}.
\end{align*}
}
\end{example}

In all the previous examples, we make no restriction on the characteristic of the base field $k$ (although in some examples one might need to avoid characteristic $2$). However, there are lots of interesting examples if the base field has positive characteristic. In \cite{WW}, Wang-Wang studied some classes of connected Hopf algebras of dimension $p^3$. It turns out that a lot of them have Poisson versions (again, unfortunately, we have no theoretical explanation). For instance, one should view the following example as a Poisson version of the Hopf algebra in \cite[Theorem 1.3(B3)]{WW}.
\begin{example}
{\rm Assume that the base field $k$ has characteristic $p>2$. Let $B=k[x,y,z]/(x^p,y^p,z^p)$ be the restricted symmetric algebra with three variables. Then $B$ becomes a commutative Hopf algebra via
\begin{align*}
\Delta(x)= &x\otimes 1+1\otimes x,\\
\Delta(y)=&y\otimes 1+1\otimes y,\\
\Delta(z)=&z\otimes 1+1\otimes z -2x\otimes y.
\end{align*}
The Poisson structure is given by
\begin{align*}
\{x,y\}=y,\ \{y,z\}=y^2,\ \{x,z\}=z.
\end{align*}
 }

\end{example}

\section{Poisson polynomial algebras}\label{poly}

The concept of a Poisson polynomial algebra was introduced in \cite{Oh3}. Roughly speaking, it is a Poisson version of the Ore extension.

Let $A$ be a Poisson algebra. A {\it Poisson derivation} on $A$ is a $k$-linear map $\alpha$ on $A$ which is a derivation with respect to both the multiplication and the Poisson bracket, that is, $\alpha(ab)=\alpha(a)b+a\alpha(b)$ and $\alpha\{a, b\}=\{\alpha(a), b\}+\{a, \alpha(b)\}$ for any $a, b\in A$. Suppose that $\delta$ is a derivation on $A$ such that 
\begin{equation}\label{alphaderivation}
\delta(\{a, b\})=\{\delta(a), b\} + \{a, \delta(b)\}+ \alpha(a)\delta(b)-\delta(a)\alpha(b)
\end{equation}
for any $a, b\in A$. Then the Poisson structure on $A$ extends uniquely to a Poisson algebra structure on the polynomial ring $R=A[x]$ such that 
\begin{equation}\label{Poissononx}
\{x, b\}=\alpha(b)x+\delta(b)
\end{equation}
for any $b\in A$. We write $R=A[x;\alpha, \delta]_p$ for this Poisson algebra and call it a {\it Poisson Ore extension} of $A$ for its similarity with Ore extensions.

Motivated by the examples in the previous section and the work in \cite{BOZZ}, we make the following definition.
\begin{definition}
Let $B$ be a Poisson Hopf algebra and let $B[x;\alpha, \delta]_p$ be a Poisson Ore extension of $B$. If $B[x;\alpha, \delta]_p$ is a Poisson Hopf algebra with $B$ as a Poisson Hopf subalgebra, then we say $B[x;\alpha, \delta]_p$ is a Poisson Hopf extension of $B$.
\end{definition}
Now one may ask the natural question: when $B[x;\alpha, \delta]_p$ is a Poisson Hopf Ore extension of $B$?

From now on, we assume that $R=B[x;\alpha, \delta]_p$ is a Poisson Ore extension of $B$ and that $\Delta(x)$ is of the following form.
$$\Delta(x)=g\otimes x+x\otimes 1+ w,$$
where $w\in B\otimes B$ and $g$ is a group-like element in $B$. Furthermore, we can assume that $\e(x)=0$ and $w\in B^+\otimes B^+$.

First, the coassociativity and the antipode axiom imply the following:
$$w\otimes 1+(\Delta\otimes Id)(w)= g\otimes w+(Id\otimes \Delta)(w),$$
and
$$S(w_1)w_2=g^{-1}w_1S(w_2).$$

The defining relation (\ref{Poissononx}) implies that $\{\Delta(x), \Delta(b)\}=\Delta\alpha(b)\Delta(x)+\Delta\delta(b)$ for any $b\in B$. 
Notice that 
\begin{align*}
&\,\,\{\Delta(x), \Delta(b)\}\\
=&\,\, \{g\otimes x+x\otimes 1+w, b_1\otimes b_2\}\\
=&\,\, gb_1\otimes \{x, b_2\}+\{g, b_1\}\otimes b_2x+\{x, b_1\}\otimes b_2+\{w, b_1\otimes b_2\}\\
=&\,\, gb_1\otimes (\alpha(b_2)x+\delta(b_2))+\{g, b_1\}\otimes b_2x+(\alpha(b_1)x+\delta(b_1))\otimes b_2+\{w, b_1\otimes b_2\}\\
=&\,\, (b_1\otimes \alpha(b_2)+\{g, b_1\}g^{-1}\otimes b_2)(g\otimes x)+ (\alpha(b_1)\otimes b_2)(x\otimes 1)+ gb_1\otimes \delta(b_2)+\delta(b_1)\otimes b_2+\{w, b_1\otimes b_2\}.
\end{align*}
As a consequence, we have
\begin{equation}\label{alphaone}
\Delta\alpha(b)=b_1\otimes \alpha(b_2)+\{g, b_1\}g^{-1}\otimes b_2,
\end{equation}
\begin{equation}\label{alphatwo}
\Delta\alpha(b)=\alpha(b_1)\otimes b_2,
\end{equation}
and 
\begin{equation}
\Delta\alpha(b)w+\Delta\delta(b)=gb_1\otimes \delta(b_2)+\delta(b_1)\otimes b_2+\{w, b_1\otimes b_2\}.
\end{equation}

Consider $\eta(b):=\alpha(b_1)S(b_2)$. By using (\ref{alphatwo}), 
$$\Delta\eta(b)=\Delta\alpha(b_1)\Delta S(b_2)=\alpha(b_1)S(b_4)\otimes b_2S(b_3)=\eta(b)\otimes 1.$$
Hence $\eta(b)$ is a scalar for any $b\in B$. Therefore, $\eta$ can be viewed as a $k$-linear map from $B$ to $k$.

Moreover, $\alpha$ can be recovered from $\eta$ by the formula 
\begin{equation}\label{alphaformula}
\alpha(b)=\eta(b_1)b_2
\end{equation}

Now (\ref{alphaone}) becomes 
$$\Delta\alpha(b)=b_1\otimes \eta(b_2)b_3+\{g, b_1\}g^{-1}\otimes b_2.$$
Applying $Id\otimes \e$ on both sides of the previous equation, one gets that
\begin{equation}
\alpha(b)= \eta(b_1)b_2= b_1\eta(b_2)+\{g, b\}g^{-1}.
\end{equation}

Since $\alpha$ is a Poisson derivation, it is easy to check that 
\begin{equation}\label{etamulti}
\eta(ab)=\e(a)\eta(b)+\e(b)\eta(a).
\end{equation}
By using (\ref{alphaformula}) and (\ref{etamulti}), it yields that
\begin{align}
\alpha(\{a, b\})&=\eta(a_1b_1)\{a_2, b_2\}+\eta(\{a_1, b_1\})a_2b_2\\
\nonumber &=(\e(a_1)\eta(b_1)+\e(b_1)\eta(a_1))\{a_2, b_2\}+\eta(\{a_1, b_1\})a_2b_2\\
\nonumber &= \{\e(a_1)a_2, \eta(b_1)b_2\}+\{\eta(a_1)a_2, \e(b_1)b_2\}+\eta(\{a_1, b_1\})a_2b_2\\
\nonumber &= \{\alpha(a), b\}+\{a, \alpha(b)\}+ \eta(\{a_1, b_1\})a_2b_2.
\end{align}
Since $\alpha$ is a derivation with respect to the Poisson bracket, we have
$$\eta(\{a_1, b_1\})a_2b_2=0$$
for any $a, b$. Applying $\e$ to the above equality, one gets
\begin{equation}
\eta(\{a_1, b_1\})\e(a_2)\e(b_2)=\eta(\{a_1\e(a_2), b_1\e(b_2)\})=\eta(\{a, b\})=0
\end{equation}
for any $a, b\in B$.

Gathering things together, we have the following theorem, which can be viewed as a Poisson version of \cite[Thereom 2.4]{BOZZ}.

\begin{theorem}\label{PoissonHopf}
Let $B$ be a Poisson Hopf algebra and let $B[x;\alpha, \delta]_p$ be a Poisson polynomial algebra which is a Poisson Hopf algebra with $B$ as a Poisson Hopf subalgebra.
Suppose $\e(x)=0$ and $\Delta(x)=g\otimes x+x\otimes 1+w$ for some $w\in B^+\otimes B^+$ and group-like element $g$.
Then
\begin{enumerate}
\item There exists a linear map $\eta: B\rightarrow k$ such that 
\begin{equation}\label{alpha}
\alpha(b)= \eta(b_1)b_2= b_1\eta(b_2)+\{g, b\}g^{-1}
\end{equation}
for any $b\in B$. The map $\eta$ satisfies the condition
\begin{equation}\label{deri}
\eta(ab)=\e(a)\eta(b)+\e(b)\eta(a)\quad\text{and}\quad\eta(\{a, b\})=0,
\end{equation}
for any $a, b\in B$. 

\item The map $\delta$ satisfies the condition
\begin{equation}\label{delta}
\Delta\delta(b)-\delta(b_1)\otimes b_2-gb_1\otimes\delta(b_2)=\{w, \Delta(b)\}-\Delta\alpha(b)w.
\end{equation}

\item The element $w=\sum w_1\otimes w_2$ satisfies the following identities
\begin{equation}\label{S}
S(w_1)w_2=g^{-1}w_1S(w_2),
\end{equation}
and
\begin{equation}\label{w}
w\otimes 1+(\Delta\otimes Id)(w)= g\otimes w+(Id\otimes \Delta)(w).
\end{equation}
\end{enumerate}
\end{theorem}

Notice that the condition $\eta(ab)=\e(a)\eta(b)+\e(b)\eta(a)$ in $(\ref{deri})$ just says that $\eta: B\rightarrow k$ is a derivation where $k$ is viewed as the trivial module over $B$.
It is not hard to show that the converse of the previous theorem is also true. Namely,
\begin{proposition}
Let $B$ be a Poisson Hopf algebra. Suppose that $\alpha$ is a Poisson derivation on $B$, $\delta$ is a derivation on $B$ satisfying $(\ref{alphaderivation})$, $g\in G(B)$ and $w\in B^+\otimes B^+$. If all these data satisfy conditions $(\ref{alpha})-(\ref{w})$, then the Poisson polynomial $B[x;\alpha, \delta]_p$ has a unique Poisson Hopf algebra structure with $B$ as a Poisson Hopf subalgebra such that $\epsilon(x)=0$ and $\Delta(x)=g\otimes x+x\otimes 1+w$.
\end{proposition}
\begin{example} 
{\rm Let $C$ be the Poisson Hopf algebra defined in Example \ref{typeA}. Clearly, $B=k[g^{\pm1}]$ is a Poisson Hopf subalgebra of $C$ with the trivial Poisson structure. Now $C$ is isomorphic to $B[x;\alpha, \delta]_p$, where $\alpha(g)=\lambda g$ and $\delta(g)=0$. In fact, if we retain the notation in Theorem \ref{PoissonHopf}, then $\eta: B\rightarrow k$ is the unique derivation such that $\delta(g)=\lambda$ and $w=0$ and all conditions in the theorem are trivially satisfied.}
\end{example}
Now we define an \emph{iterated Poisson Hopf Ore extension} 
to be a Poisson Hopf algebra $B$ containing a chain of Poisson Hopf subalgebras
\begin{equation} 
\label{introdef1} 
k = B_{(0)} \subset \cdots \subset B_{(i)} \subset B_{(i+1)} 
\subset \cdots \subset B_{(n)} = B, 
\end{equation}
with each of the extensions $B_{(i)} \subset B_{(i+1)} := 
B_{(i)}[x_{i+1}; \sigma_{i+1},\delta_{i+1}]$ being a Poisson Hopf Ore extension. 

It is not hard to see that both Poisson Hopf algebras in Example \ref{GK3} and Example \ref{GK4} can be realized as iterated Poisson Hopf Ore extensions.

\medskip
\section{Universal enveloping algebras of Poisson algebras}\label{Puniversal}
Let $A$ be a Poisson algebra. Then one can define the universal enveloping algebra of $A$ (See \cite{U}). We briefly recall the construction here. Let $m_A= \{m_a : a\in A\}$ and $h_A=\{h_a : a\in A\}$ be two copies of the vector space $A$ endowed with two $k$-linear isomorphisms $m : A\rightarrow m_A$ sending $a$ to $m_a$ and $h : A\rightarrow h_A$ sending $a$ to $h_a$.  Then $A^e$ is an associative algebra over $k$, with an identity $1$, generated by $m_A$ and $h_A$ with relations
\begin{align}
\label{5.1}m_{xy}=m_xm_y,\\
h_{\{x, y\}}=h_xh_y-h_yh_x,\\
\label{hxy}h_{xy}=m_yh_x+m_xh_y,\\
m_{\{x, y\}}=h_xm_y-m_yh_x=[h_x, m_y],\\
\label{5.5}m_1=1.
\end{align}

The following result is known. See \cite[Corollary 1]{U}.
\begin{proposition}
The category of left (resp. right) Poisson modules over a Poisson algebra $A$ is equivalent to the left (resp. right) module category over $A^e$.
\end{proposition}

It turns out that the map $m$ and $h$ induce an algebra map and a Lie map, respectively. Moreover, the map $m$ is injective \cite[Section 2]{U}. Therefore we can identify $x\in A$ with $m_x\in A^e$. We will sometimes refer to $m$ and $h$ as {\it structure maps} of $A^e$.

Obviously, $A$ is an $A^e$-module via the action defined by $h_x\cdot a=\{x, a\}$ and $m_x\cdot a=xa$ for any $x, a\in A$. There is a unique $A^e$-module map $\xi: A^e\rightarrow A$ sending $1$ to $1$.

An equivalent definition for $A^e$ is to use the universal property of the structure maps $m$ and $h$ \cite[Definition 1]{Oh2}.

\subsection{Poisson derivation}
Let $A$ be a Poisson algebra and $M$ a Poisson $A$-module. A $k$-linear map $\delta: A\rightarrow M$ is called a {\it Poisson derivation} with coefficients in $M$ if 
$$\delta(ab)=a\delta(b)+b\delta(a),$$
and $$\delta(\{a, b\})=\{a, \delta(b)\}-\{b, \delta(a)\}.$$
A {\it universal Poisson derivation} of the Poisson algebra $A$ is a Poisson module $\Theta_A$ together with a Poisson derivation $D: A\rightarrow \Theta_A$ such that for any Poisson derivation $\delta: A\rightarrow M$ there is a unique Poisson module map $f: \Theta_A\rightarrow M$ with $\delta=f\circ D$. As shown in \cite[Lemma 1]{U}, $\Theta_A$ can be constructed as the left ideal $I$ of $A^e$ generated by all $h_x$ where $x\in A$ with the derivation $D$ given by 
$$D: A\rightarrow I\quad (x\mapsto h_x).$$
Based on this construction, we have the following proposition.

\begin{proposition}\label{decomp}
Let $A$ be a Poisson algebra. Then $A^e\cong A\oplus I$ as left $A$-modules, where $I$ is the left ideal of $A^e$ generated by all $h_x$ with $x\in A$. Consequently,
$$A^e\cong A\oplus\Theta_A$$
as left $A$-modules, where $\Theta_A$ is the universal Poisson derivation of $A$.
\end{proposition}
\begin{proof}
We have seen that there is a unique $A^e$-module map $\xi: A^e\rightarrow A$ sending $1$ to $1$. Clearly, $\xi \circ m=id_A$. Since both $\xi$ and $m$ are $A$-module maps, $A^e$ is isomorphic to $A\oplus I$ as $A$-modules, where $I$ is the kernel of the map $\xi$. It is easy to show that $I$ agrees with the left ideal of $A^e$ generated by all $h_x$ where $x\in A$. Now the result follows from \cite[Lemma 1]{U}. 
\end{proof}

\subsection{Relation with Lie-Rinehart algebra}
The notion of Lie-Rinehart algebras appeared in several literatures under different names before the first thorough study by Rinehart in \cite{Ri}. It can be viewed as the algebraic analogue of the notion of Lie algebroids. In this section, we are going to study the relation between $A^e$ and the Lie-Rinehart algebra structure on $\Omega_A$, the K\"alher differentials of $A$.

\begin{definition}\label{LR}
Let $R$ be a commutative ring with identity, $A$ a commutative $R$-algebra and $L$ a Lie algebra over $R$. The pair $(A, L)$ is called a {\it Lie-Rinehart algebra} over $A$ if $L$ is a left $A$-module and there is an {\it anchor map} $\rho: L\rightarrow \Der_R(A)$, which is an $A$-module and a Lie algebra morphism, such that the following relation is satisfied,
\begin{equation}
[\xi, a\cdot\zeta]= a\cdot [\xi, \zeta]+ \rho(\xi)(a)\cdot \zeta,
\end{equation}
for any $a\in A$ and $\xi, \zeta\in L$.
\end{definition}

For simplicity, we will use $\xi(a)$ for $\rho(\xi)(a)$ where $\xi\in L$ and $a\in A$. From now on, if not mentioned otherwise, we would assume that $R$ is the ground field $k$.  The following example is the one that we are mainly interested \cite[Theorem 3.8]{H}.
\begin{example}\label{differential}
{\rm Let $A$ be a Poisson algebra over $k$ and $\Omega_A$ its K\"ahler differentials. Then the pair $(A, \Omega_A)$ becomes a Lie-Rinehart algebra over $k$ where the Lie bracket on $\Omega_A$ is given by 
\begin{equation}
[a df, bdg]= abd\{f, g\}+ a\{f, b\}dg-b\{g, a\}df,
\end{equation}
and the anchor map sends $df$ to $\{f, \cdot \}$.}
\end{example}

Next we recall the construction of $V(A, L)$, the enveloping algebra of a Lie-Rinehart algebra $(A, L)$ over $R=k$. Via the map $\rho: L\rightarrow \Der_kA$, $A$ becomes a Lie module over the Lie algebra $L$. Let $\mathfrak{h}=A\rtimes L$, the semidirect product of $A$ with $L$ \cite[7.4.9]{We} and take $U=U(\mathfrak{h})$. Recall that as a $k$-space, $A\rtimes L$ is just the direct sum of $A$ and $L$ and the Lie bracket is given by
$$[a+X, b+Y]=(X(b)-Y(a))+[X, Y],$$
where $a, b\in A$ and $X, Y\in L$.
Also, as the direct sum of $A$ and $L$, $A\rtimes L$ carries a natural $A$-module structure. For any $z\in A\rtimes L$, denote by $z'$ the canonical image of $z$ in $U$. Let $P$ be the two-sided ideal generated by all elements of the form $(a.z)'-a'z'$, where $a\in A$ and $z\in A\rtimes L$. Then $V(A, L)$ is defined to be $U/P$ and there is a canonical map $A\rtimes L\rightarrow V(A, L)$.

Now let $(A, \Omega_A)$ be the Lie-Rinehart algebra defined in Example \ref{differential}, where $A$ is a Poisson algebra. Then we have two maps
\begin{equation}
\alpha: A\rightarrow A\rtimes\Omega_A\rightarrow V(A, \Omega_A),
\end{equation}
and 
\begin{equation}
\beta: A\xrightarrow{d} \Omega_A\rightarrow A\rtimes\Omega_A\rightarrow V(A, \Omega_A).
\end{equation}
\begin{lemma}\label{structuremaps}
Retain the above notation. Then $\alpha$ is an algebra map and $\beta$ is a Lie algebra map. Moreover, we have 
\begin{equation}
\alpha(\{a, b\})=[\beta(a), \alpha(b)], \,\, \beta(ab)=\alpha(a)\beta(b)+\alpha(b)\beta(b)
\end{equation}
for any $a, b\in A$.
\end{lemma}
\begin{proof} Recall that the anchor map $\rho$ sends $da$ to $\{a, \cdot\}$. Hence by all previous convention $da(b)=\rho(da)(b)=\{a, b\}$ for any $a, b\in A$. Therefore,
\begin{equation*}
[\beta(a), \alpha(b)]= [da, b]= da(b)=\{a, b\}= \alpha(\{a, b\}).
\end{equation*}
The second equation simply amounts to 
\begin{equation*}
d(ab)=adb+bda,
\end{equation*}
which is obviously true.
\end{proof}

\begin{proposition}\label{universal}
Retain the notation in Lemma \ref{structuremaps}. Let $B$ be a $k$-algebra, $\gamma: A\rightarrow B$ an algebra map and $\delta: A\rightarrow B$ a Lie algebra map such that 
\begin{equation}
\gamma(\{a, b\})=[\delta(a), \gamma(b)], \,\, \delta(ab)=\gamma(a)\delta(b)+\gamma(b)\delta(b)
\end{equation}
for any $a, b\in A$. Then there exists a unique algebra map $\lambda: V(A, \Omega_A)\rightarrow B$ such that the diagram

\[
\begin{CD}
A@>\alpha>> V(A, \Omega_A)@<\beta<< A\\
@V=VV   @V\lambda VV   @VV=V\\
A@>\gamma>> B@<\delta<< A\\
\end{CD}
\]
commutes.

\end{proposition}
\begin{proof}
As in the construction of the enveloping algebra of a Lie-Rinehart algebra, we identify $V(A, \Omega_A)$ with $U/P$, where $U=U(A\rtimes \Omega_A)$ and $P$ is the two-sided ideal generated by all elements of the form $(a.z)'-a'z'$ with $a\in A$ and $z\in A\rtimes \Omega_A$.

First, the algebra $B$ becomes an $A$-module via the algebra map $\gamma$, i.e. $a\cdot m:=\gamma(a)m$ for any $a\in A$ and $m\in B$. Now the condition $\delta(ab)=\gamma(a)\delta(b)+\gamma(b)\delta(b)$ just says that $\delta$ is a derivation. Therefore there is an $A$-module map $\theta: \Omega_A\rightarrow B$ such that $\delta=\theta\circ d$. Consequently, we have a well-defined $A$-linear map from $A\rtimes \Omega_A$ to $B$ sending $a+bdf$ to $\gamma(a)+b\cdot\delta(f)=\gamma(a)+\gamma(b)\delta(f)$. This map is a Lie algebra map and thus induces a map $U\rightarrow B$ which factors through $U/P=V(A, \Omega_A)$. Denote this map by $\lambda: V(A, \Omega_A)\rightarrow B$. The commutativity of the diagram is clear from the construction.
\end{proof}

\begin{proposition}\label{iso}
Let $A$ be a Poisson algebra. Then $A^e\cong V(A, \Omega_A)$. To be precise, there is a unique isomorphism $\Lambda$ such that the diagram
\[
\begin{CD}
A@>\alpha>> V(A, \Omega_A)@<\beta<< A\\
@V=VV   @V\Lambda VV   @VV=V\\
A@>m>> A^e@<h<< A\\
\end{CD}
\]
 commutes.
\end{proposition}
\begin{proof}
This is a consequence of Proposition \ref{universal} and the universal property of $A^e$ \cite[Definition 1]{Oh2}.
\end{proof}

In \cite{Ri}, a filtration is defined on $V(A, L)$ for any Lie-Rinehart algebra $(A, L)$ over $k$. In particular, in our case where $A$ is a Poisson algebra, the algebra $V(A, \Omega_A)$ carries a filtration which naturally passes to $A^e$ via the isomorphism $\Lambda$ as in Proposition \ref{iso}. The filtration $\{F_n\}_{n\ge 0}$ is such that $A\subset F_0 A^e$ and $h_x\in F_1A^e$ for any $x\in A$. Now the theorem \cite[Theorem 3.1]{Ri} says the following.

\begin{corollary}\label{PBW}
Let $A$ be a Poisson algebra. If the K\"ahler differential $\Omega_A$ is a projective $A$-module, then there is an $A$-algebra isomorphism 
\begin{equation}
S_A(\Omega_A)\cong \gr_F A^e,
\end{equation}
where $S_A(\Omega_A)$ is the symmetric $A$-algebra on $\Omega_A$.
\end{corollary}

Corollary \ref{PBW} can be viewed as a Poincar\'e-Birkhoff-Witt theorem for the enveloping algebra of a Poisson algebra. In particular, it recovers the result in \cite{OPS}.

\begin{remark}\label{generalcase}
In general, the canonical $A$-module map $S_A(\Omega_A)\rightarrow \gr_F A^e$ is surjective. So at least we know that as an $A$-module, $A^e$ is generated by elements of the form 
$$h_{a_1}\cdots h_{a_s},$$
where $a_i\in A$ and $s$ is any integer.
\end{remark}

\subsection{Poisson cohomology and homology.} 
The relation of Poisson (co)homology with $V(A, \Omega_A)$ is studied in \cite{Ri} and \cite{H}. In view of Proposition \ref{iso}, we can reformulate that by using the Poisson enveloping algebra $A^e$. 

Let $A$ be a Poisson algebra. Consider the complex $C_*$ where $C_n=0$ for $n<0$ and $C_n=A^e\otimes_A\Omega_{A/k}^n$ for $n\ge 0$. Conventionally, we take $\Omega_{A/k}^0=A$. The differential $b$ is given by
\begin{align*}
b(a_0\otimes \de a_1\de a_2\cdots \de a_{n})&=\sum_{i=1}^{n}(-1)^{i+1}a_0h_{a_i}\otimes \de a_1\de a_2\cdots\hat{\de a_i} \cdots\de a_{n}\\
&+ \sum_{1\le i<j\le n}(-1)^{i+j}a_0\otimes \de \{a_i, a_j\} \de a_1\cdots \hat{\de a_i} \cdots\hat{\de a_j} \cdots\de a_n.
\end{align*}

\begin{proposition} Let $A$ be a Poisson algebra and suppose that $\Omega_A$ is projective over $A$. Then the complex $C_*$ defined above is a projective resolution of $A$ as a left $A^e$-module.
\end{proposition}
\begin{proof}
This is a consequence of Proposition \ref{iso} and \cite[Lemma 4.1]{Ri}.
\end{proof}

For a Poisson module $M$ over a Poisson algebra $A$, let $\mathfrak{X}^s(A, M)=\Hom_A(\Omega^s_{A}, M)$. The Poisson coboundary operator $\delta^i: \mathfrak{X}^s(A, M)\rightarrow \mathfrak{X}^{s+1}(A, M)$ is defined, for any $Q\in \mathfrak{X}^s(A, M)$, by
$$\delta^i(Q)(f_0, ..., f_s)=\sum_{i=0}^{s}(-1)^i\{f_i, Q(f_0, ..., \hat{f_i}, ...f_s)\}_M+\sum_{0\le i<j\le s}(-1)^{i+j}Q(\{f_i, f_j\}, f_0, ..., \hat{f_i}, ..., \hat{f_j},...).$$
The cohomology of the complex $\mathfrak{X}^*(A, M)$ is called the Poisson cohomology with coefficients in $M$ and denoted by $\HP^*(A, M)$. If $M=A$, then $\HP^*(A)$ is used in stead of $\HP^*(A, A)$. Similarly, one can define $\HP_*(A, N)$, the Poisson homology with coefficients in $N$, where $N$ is a right Poisson module over $A$ (See, for example, \cite{ZVZ}).

In \cite{H}, Huebschmann proved that, when $\Omega_A$ is projective over $A$, 
\begin{equation*}
\HP^*(A, M)\cong \Ext^*_{V(A, \Omega_A)}(A, M),
\end{equation*}
and 
\begin{equation*}
\HP_*(A, N)\cong \Tor_*^{V(A, \Omega_A)}(N, A),
\end{equation*}
where $M$ is a left Poisson module and $N$ a right Poisson module. Consequently, 
\begin{proposition}
Let $A$ be a Poisson algebra and suppose that $\Omega_A$ is projective over $A$. Then
\begin{equation}
\HP^*(A, M)\cong \Ext^*_{A^e}(A, M),
\end{equation}
and 
\begin{equation}
\HP_*(A, N)\cong \Tor_*^{A^e}(N, A),
\end{equation}
for any left Poisson module $M$ and any right Poisson module $N$.
\end{proposition}

\medskip
\section{Hopf structures on Poisson enveloping algebras}\label{Hopf}
In \cite{Oh2}, it is shown that if $B$ is a Poisson bialgebra, then $B^e$ carries a bialgebra structure as well.
\begin{theorem}\label{Hopfstr}\cite[Theorem 10]{Oh2}
Let $B$ be a Poisson bialgebra with comultiplication $\Delta$ and counit $\e$. Then $B^e$ has a bialgebra structure with comultiplication $\Delta^e$ and counit $\e^e$ such that
$$\Delta^em=(m\otimes m)\Delta, \,\,\,\, \e^em=\e,$$
and 
$$\Delta^eh=(m\otimes h+h\otimes m)\Delta, \,\,\,\, \e^eh=0.$$
Morever, if $B$ is a Hopf algebra with antipode $S$, then $B^e$ is a Hopf algebra with antipode $S^e$ such that 
$$S^em=mS, \,\,\,\, S^eh=hS.$$
\end{theorem}

From now on, we will drop the superscript $e$ for the comultiplication, counit and antipode on $B^e$ if there is no confusion. The following proposition is expected.
\begin{proposition}
The category of Poisson modules over a Poisson bialgebra $B$ is equivalent to the module category over $B^e$ as monoidal categories.
\end{proposition}

\subsection{The Poisson enveloping algebra of $PS(\mathfrak{g})$}
The Poisson enveloping algebra of $PS(\mathfrak{g})$ is calculated in \cite[Example 11]{Oh2}. Here we are going to give a more Hopf-theoretical description.

Let $\mathfrak{g}$ be a Lie algebra and let $V$ be a copy of $\mathfrak{g}$ as a $k$-space. Then $V$ becomes a $\mathfrak{g}$-module via the adjoint action, which naturally extends to $S(V)$. In fact, $S(V)$ becomes a $U(\mathfrak{g})$-module algebra in the sense of \cite[Definition 4.1.1]{Mo}. Therefore one can form the smash product $S(V)\#U(\mathfrak{g})$.

It is easy to check that $S(V)\#U(\mathfrak{g})\cong U(V\rtimes \mathfrak{g})$, where $V\rtimes \mathfrak{g}$ is the semidirect product of $V$ with $\mathfrak{g}$ \cite[7.4.9]{We}. First, there is a Lie algebra map from $V\rtimes \mathfrak{g}$ to $S(V)\#U(\mathfrak{g})$ sending $(v, g)$ to $v\#1+1\#g$. By the universal property of enveloping algebra, this map induces an algebra map $U(V\rtimes \mathfrak{g})\rightarrow S(V)\#U(\mathfrak{g})$, which is actually an isomorphism.

\begin{proposition}
Suppose that the base field $k$ has characteristic $0$. Let $B$ be the Poisson Hopf algebra $PS(\mathfrak{g})$. Then $B^e\cong U(V\rtimes \mathfrak{g})$ as Hopf algebras.
\end{proposition}
\begin{proof}
By the defining relation of $B^e$, it is easy to check that there is a Lie algebra map $V\rtimes \mathfrak{g}\rightarrow B^e$ sending $(x, y)$ to $x+h_y$. This map induces an algebra map $\Phi: U(V\rtimes \mathfrak{g})\rightarrow B^e$. To see that $\Phi$ is a Hopf algebra map, one just have to notice that $x+h_y$ is a primitive element by Theorem \ref{Hopfstr}. Now it follows from Proposition \ref{PBW} or \cite[Theorem 7.3]{OPS} that $\Phi$ is an isomorphism.
\end{proof}

\subsection{Some general properties}

As one can see from Theorem \ref{Hopfstr}, the injective map $m$ preserves the Hopf algebra structure. Hence, if $B$ is a Poisson Hopf algebra, we can identify $B$ as a Hopf subalgebra of $B^e$.

\begin{proposition}\label{normalsub}
Let $B$ be a Poisson Hopf algebra. Then $B$ is a normal Hopf subalgebra of $B^e$.
\end{proposition}

\begin{proof}
For the definition of normal Hopf subalgebra, see \cite[Definition 3.4.1]{Mo}. 
By Theorem \ref{Hopfstr}, one gets
$$\Delta(h_y)= y_1\otimes h_{y_2}+ h_{y_1}\otimes y_2
$$
for any $y\in B$.
By Remark \ref{generalcase}, every element in $B^e$ is a linear combination of elements of the form 
$xh_{y_1}\cdots h_{y_n}$
for some $x, y_j\in B$.
Therefore, it suffices to show that 
$$(ad_l h_y)(x)\in B \quad\text{and} \quad(ad_r h_y)(x)\in B,$$
for any $x, y\in B$, where $ad_l$ and $ad_r$ are defined in \cite[definition 3.4.1]{Mo}.
Now,
\begin{align*}
(ad_l h_y)(x)&= y_1xh_{Sy_2}+ h_{y_1}xSy_2\\
&=xy_1h_{Sy_2}+xh_{y_1}Sy_2+\{y_1, x\}Sy_2\\
&=x\e(h_y)+\{y_1, x\}Sy_2\\
&=\{y_1, x\}Sy_2\in B
\end{align*}
Similarly, one can show that
\begin{equation*}
(ad_r h_y)(x)=\{Sy_1, x\}y_2\in B.
\end{equation*}
This completes the proof.
\end{proof}

\begin{corollary}
Let $B$ be a Poisson Hopf algebra. Then $B^+B^e=B^eB^+$ and $I=B^eB^+$ is a Hopf ideal of $B^e$.
\end{corollary}
\begin{proof}
This follows from Proposition \ref{normalsub} and \cite[Lemma 3.4.2(1)]{Mo}.
\end{proof}
For the sake of simplicity,  we put 
$$\mathcal{H}(B):=B^e/B^eB^+$$ in the rest of the paper.

\begin{proposition}\label{coradicalBE}
Let $B$ be a Poisson Hopf algebra. Then the coradical of $B^e$ is equal to the coradical of $B$.
\end{proposition}
\begin{proof}
First we define a filtration $\{\mathcal{F}_n\}_{n\ge 0}$ on $B^e$. Let $\mathcal{F}_0$ be the subalgebra of $B^e$ generated by $B_0$, the coradical of $B$. Since the antipode of $B$ is bijective, $\mathcal{F}_0$ is a Hopf subalgebra of $B^e$. Define $\mathcal{F}_n$ to be 
$$\mathcal{F}_n=\wedge^{n+1}\mathcal{F}_0,$$
where the $\wedge$ is taken in $B^e$. Apparently, $\{\mathcal{F}_n\}_{n\ge 0}$ is an ascending chain of $k$-spaces and $B_{[n]}\subset \mathcal{F}_n$ for any $n$. We are going to prove that $\{\mathcal{F}_n\}_{n\ge 0}$ is a Hopf filtration on $B^e$. Similar to the proof of \cite[Theorem 5.2.2 (2)]{Mo}, one has
$$\Delta{\mathcal{F}_n}\subset \sum_{i=0}^n\mathcal{F}_i\otimes \mathcal{F}_{n-i}.$$
Moreover, it is easy to show by induction that
\begin{equation}\label{algebrafil}
\mathcal{F}_n\mathcal{F}_m\subset \mathcal{F}_{n+m}.
\end{equation}
Now it only remains to show that $\{\mathcal{F}_n\}_{n\ge 0}$ is exhaustive, i.e. $\bigcup_{i=0}^{\infty}\mathcal{F}_i=B^e$.
By the construction, $B^e$ is generated by elements of the form $a$ and $h_b$, where $a, b\in B$, as an algebra. Hence by (\ref{algebrafil}), it suffices to show that for any $b\in B$, $h_b\in \mathcal{F}_N$ for some $N\ge 0$. Suppose that $b\in B_{[\ell]}\subset \mathcal{F}_{\ell}$. We proceed by induction on $\ell$ and show that $N=\ell+1$.

If $\ell =0$, then $\Delta(b)=b_1\otimes b_2\in B_{[0]}\otimes B_{[0]}$. Now
$$\Delta(h_b)=h_{b_1}\otimes b_2+b_1\otimes h_{b_2}\in B^e\otimes \mathcal{F}_0+\mathcal{F}_0\otimes B^e.$$
Hence $h_b\in \Delta^{-1}(B^e\otimes\mathcal{F}_0+\mathcal{F}_0\otimes B^e)=\mathcal{F}_0\wedge\mathcal{F}_0=\mathcal{F}_1$.

Suppose $\ell \ge 1$. Then $\Delta(b)=b_1\otimes b_2\in \sum_{i=0}^\ell B_{[i]}\otimes B_{[\ell -i]}$. 
Clearly, $b_1\otimes b_2\in B\otimes B_{[\ell]}$, so $h_{b_1}\otimes b_2\in B^e\otimes B_{[\ell]}\subset B^e\otimes \mathcal{F}_{\ell}$. Also, $b_1\otimes b_2\in B_{[0]}\otimes B+\sum_{i=1}^\ell B_{[i]}\otimes B_{[\ell -i]}$. Hence by induction hypothesis, 
$b_1\otimes h_{b_2}\in B_{[0]}\otimes B^e+B^e\otimes \mathcal{F}_{\ell}$.
Consequently,
$$\Delta(h_b)=h_{b_1}\otimes b_2+b_1\otimes h_{b_2}\in B^e\otimes \mathcal{F}_{\ell}+\mathcal{F}_0\otimes B^e.$$
Now it is clear that $h_b\in \Delta^{-1}(B^e\otimes\mathcal{F}_{\ell}+\mathcal{F}_0\otimes B^e)=\mathcal{F}_{\ell+1}$.

Since $\{\mathcal{F}_n\}_{n\ge 0}$ is a Hopf algebra filtration (in particular, a coalgebra filtration) on $B^e$, the coradical of $B^e$ is contained in $\mathcal{F}_0=B_{[0]}\subset B$ by \cite[Lemma 5.3.4]{Mo}. The result then follows.
\end{proof}

This proposition has some nice consequences. For a right comodule $M$ over a coalgebra $C$, the {\it coinvariants} of $C$ in $M$ are the set $M^{coC}:= \{m\in M: \rho(m)=m\otimes 1\}$, where $\rho$ is the comodule structure map. Similarly, for a left $C$-comodule $M$, one can define $\,\,^{coC}M$.

\begin{corollary}\label{freeness}
Let $B$ be a Poisson Hopf algebra. Then the following statements are true. 
\begin{enumerate}

\item[\textup{(1)}] $B^e$ is a free module over $B$,
\item[\textup{(2)}] $B=(B^e)^{co\mathcal{H}(B)} =\,\,^{co\mathcal{H}(B)}(B^e)$,
\item[\textup{(3)}] $B^e$ is an injective comodule over $\mathcal{H}(B)$.
\end{enumerate}
\end{corollary}
\begin{proof}
Since the coradical of $B^e$ is equal to that of $B$, part (1) follows from \cite[Corollary 2.3]{Ra1}. Then part (2) and (3) follow from \cite[Theorem 1.4]{Schn}.
\end{proof}

The algebra $B^e$ is clearly an $\mathcal{H}(B)$-comodule algebra in the sense of \cite[Definition 4.1.2]{Mo}. In terms of \cite[Definition 7.2.1]{Mo}, Corollary \ref{freeness} (b) just means that $B\subset B^e$ is an $\mathcal{H}(B)$-extension.

\begin{corollary}\label{BEpointed}
Let $B$ be a pointed Poisson Hopf algebra. Then $B^e$ is a pointed Hopf algebra and $G(B)=G(B^e)$.
\end{corollary}

\subsection{The structure of $\mathcal{H}(B)$}

The purpose of this section is to explore the structure of $\mathcal{H}(B)$ where $B$ is a Poisson Hopf algebra. To start with, we have the following proposition.
\begin{proposition}\label{primitivegen}
Let $B$ be a Poisson Hopf algebra. Then $\mathcal{H}(B)$ is a connected Hopf algebra and is generated by primitive elements as an algebra.
\end{proposition}
\begin{proof}
Let $V=(B^e_0)^+$. By Propostion \ref{coradicalBE}, $V=B_0^+$ and therefore $V\subset B^+B^e$. Consequently there is a surjective coalgebra map $B^e/V\rightarrow \mathcal{H}(B)$. By \cite[Lemma 5.3.8]{Mo}, $B^e/V$ is a connected coalgebra. Hence $\mathcal{H}(B)$ is a connected coalgebra by \cite[Corollary 5.3.5]{Mo}.

Let $\pi$ be the projection from $B^e$ to $\mathcal{H}(B)$ and let $\theta$ be the composition $\pi\circ h$. By the construction of $\mathcal{H}(B)$, $\pi(a)=\e(a)$ for any $a\in B$. Since $B^e$ is generated by elements of the form $a$ and $h_b$, where $a, b\in B$, as an algebra, $\mathcal{H}(B)$ is generated by elements of the form $\theta(b)$ where $b\in B$. Now
\begin{align*}
\Delta(\theta(b))&=(\pi\otimes \pi)\Delta(h_b)\\
&=\pi(b_1)\otimes \theta(b_2)+\theta(b_1)\otimes \pi(b_2)\\
&=\e(b_1)\otimes \theta(b_2)+\theta(b_1)\otimes \e(b_2)\\
&=1\otimes \theta(b)+\theta(b)\otimes 1.
\end{align*}
This completes the proof.
\end{proof}

For any Poisson Hopf algebra $B$, the space $B^+/(B^+)^2$ is indeed a Lie algebra with Lie bracket induced by the Poisson structure on $B$.

\begin{proposition}\label{Lie}
 Suppose that the base field $k$ has characteristic $0$. Let $B$ be a finitely generated Poisson Hopf algebra and let $\mathfrak{m}=B^+$. Then $\mathcal{H}(B)\cong U(\mathfrak{a})$ as Hopf algebras, where $\mathfrak{a}$ is the Lie algebra $\mathfrak{m}/\mathfrak{m}^2$.
\end{proposition}
\begin{proof}
Let $\theta$ be the composition $\pi\circ h$, where $\pi: B^e\rightarrow \mathcal{H}(B)$ is the canonical projection. Let $G$ be the algebraic group with $\mathcal{O}(G)=B$. By a well-known result of Cartier, $G$ is smooth. Hence $\Omega_B$ is a free module over $B$ of rank equal to $\ell=\dim_k \mathfrak{m}/\mathfrak{m}^2$. Let $\{x_1, x_2, ..., x_\ell\}$ be a set of elements in $\mathfrak{m}$ whose image in $\mathfrak{m}/\mathfrak{m}^2$ form a $k$-basis. By Corollary \ref{PBW}, $B^e$ is a free $B$-module with a $B$-basis consisting of elements of the form
\begin{equation}\label{Bbasis}
h_{x_1}^{i_1}\cdots h_{x_\ell}^{i_\ell}.
\end{equation}
Let $W$ be the $k$-space of $B^e$ spanned by elements in $(\ref{Bbasis})$. Then $B^+B^e=B^+W$. Consequently, $\pi$ restricts to a $k$-space isomorphism from $W$ to $\mathcal{H}(B)$. In particular, $\mathcal{H}(B)$ has a $k$-basis consisting of elements of the form 
\begin{equation}\label{kbasis}
\theta(x_1)^{i_1}\cdots \theta(x_\ell)^{i_\ell}.
\end{equation}
Since $\ch k=0$, Proposition \ref{primitivegen} tells that $\mathcal{H}(B)$ is isomorphic to $U(L)$, where $L$ is the primitive space of the connected Hopf algebra $\mathcal{H}(B)$. By the proof of Proposition \ref{primitivegen}, $\{\theta(x_1),...., \theta(x_\ell)\}\subset L$. Moreover, from the $k$-basis $(\ref{kbasis})$ of $\mathcal{H}(B)$, we see that $\{\theta(x_1),...., \theta(x_\ell)\}$ is a $k$-basis for $L$. Therefore, the Lie algebra morphism $\theta$ maps $\mathfrak{m}$ onto $L$. It is easy to check that $k1+\mathfrak{m}^2\subset \ker \theta$. Since $\dim_kL=\dim_k \mathfrak{m}/\mathfrak{m}^2$, $\ker \theta =k 1+ \mathfrak{m}^2$ and $\theta$ induces a Lie algebra isomorphism from $\mathfrak{m}/\mathfrak{m}^2$ to $L$. This completes the proof.
\end{proof}

In fact, the Lie algebra $\mathfrak{m}/\mathfrak{m}^2$ carries more structure - it is actually a so-called Lie bialgebra, whose definition we briefly recall here.

Let $V$ be a vector space. Let $\Sym(V\otimes V)$ be the subspace of $V\otimes V$ spanned by elements of the form $x\otimes y+y\otimes x$ and let $\Asym(V\otimes V)$ be the subspace of $V\otimes V$ spanned by elements of the form $x\otimes y-y\otimes x$. Clearly, we have
$$V\otimes V=\Sym(V\otimes V)\oplus \Asym(V\otimes V).$$

\begin{definition}\label{Liecoalgebra}
A Lie coalgebra $L$ is a vector space with a $k$-linear map $\delta: L\rightarrow \Asym(L\otimes L)$ such that 
\begin{equation}
(1+\xi+\xi^2)\circ(1\otimes \delta)\circ \delta=0,
\end{equation}
where $\xi: L^{\otimes 3}\rightarrow L^{\otimes 3}$ is the $k$-linear map induced by the cyclic permutation $x\otimes y\otimes z\mapsto y\otimes z\otimes x$.
\end{definition}

Obviously, any coalgebra $C$ with comultiplication $\Delta$ becomes a Lie coalgebra via $\delta:= \Delta- \Delta^{op}$.

\begin{definition}\label{Liebialgebra}
A Lie bialgebra $L$ is a Lie algebra with a Lie coalgebra structure given by a cobracket $\delta$ such that 
\begin{equation}\label{Liebialgebra}
\delta([a, b])=a\cdot\delta(b)-b\cdot\delta(a),
\end{equation}
where 
$$a\cdot(b\otimes c)=[a, b]\otimes c+b\otimes [a, c].$$
\end{definition}

\begin{remark}\label{coPoisson}\cite[Proposition 2.2.10]{KS} 
For a Lie bialgebra $L$, the cobracket $\delta$ naturally extends to a map $U(L)\rightarrow U(L)\otimes U(L)$, which makes $U(L)$ a coPoisson Hopf algebra in the sense of \cite[Definition 2.2.9]{KS}.
\end{remark}

The following argument is taken from \cite{EK}, with a slightly different notation.
Now, assume that the base field has characteristic $0$. Let $\mathfrak{a}= \mathfrak{m}/\mathfrak{m}^2$, where $\mathfrak{m}=B^+$. It is clear that the Poisson bracket extends to the completion of $B$ with respect to the $\mathfrak{m}$-adic topology
$$\lim_{\leftarrow}B/\mathfrak{m}^n,$$
which is exactly $k[[\mathfrak{a}]]$. Therefore, $k[[\mathfrak{a}]]$ becomes a topological Poisson Hopf algebra. Then the space of continuous functions on $k[[\mathfrak{a}]]$ is the universal enveloping algebra $U(\mathfrak{a}^*)$ of $\mathfrak{a}^*$. Moreover, $U(\mathfrak{a}^*)$ becomes a coPoisson Hopf algeba with the cobracket $\delta$ given by the dual of the Poisson bracket on $k[[\mathfrak{a}]]$. It is easy to show that $\delta(\mathfrak{a}^*)\subset \mathfrak{a}^*\otimes \mathfrak{a}^*$, and $\delta$ defines a Lie bialgebra structure on $\mathfrak{a}^*$. Consequently, $\mathfrak{a}$ carries a natural Lie bialgebra structure as the dual of $\mathfrak{a}^*$. To summarize, we have
\begin{proposition}
Suppose that the base field $k$ has characteristic $0$. Let $B$ be a finitely generated Poisson Hopf algebra and let $\mathfrak{m}=B^+$. Then $\mathfrak{a}=\mathfrak{m}/\mathfrak{m}^2$ is a Lie bialgebra as described above. Consequently, $\mathcal{H}(B)$ is a coPoisson Hopf algebra.
\end{proposition}
\begin{proof}
We have already shown that $\mathfrak{a}=\mathfrak{m}/\mathfrak{m}^2$ is a Lie bialgebra. Now it follows from Proposition \ref{Lie} and Remark \ref{coPoisson} that $\mathcal{H}(B)$ is a coPoisson Hopf algebra.
\end{proof}

In fact, if the Poisson Hopf algebra $B$ is connected in the sense that its coradical is one-dimensional, the coPoisson Hopf algebra structure on $\mathcal{H}(B)$ is easier to describe. As shown in Proposition \ref{Lie}, $\mathcal{H}(B)=U(L)$ where $L$ is the primitive space of $\mathcal{H}(B)$. Therefore, all we have to do is to describe the Lie bialgebra structure on $L$.

Recall that $B$ becomes a Lie coalgebra with $\delta=\Delta-\Delta^{op}$.

\begin{lemma}\label{cobracket}
Suppose that the base field $k$ has characteristic $0$. Let $B$ be a finitely generated connected Poisson Hopf algebra and let $J:=\ker\theta$. Then
$$\delta(J)\subset B\otimes J+ J\otimes B,$$
where $\delta=\Delta-\Delta^{op}$. Consequently, $B/J= \mathfrak{m}/\mathfrak{m}^2$ is a quotient Lie coalgebra of $B$, where $\mathfrak{m}=B^+$.
\end{lemma}
\begin{proof}
 As shown in the proof of Proposition \ref{Lie}, $J=k1+\mathfrak{m}^2$. Let $T=\mathfrak{m}^2$ and we only need to show that 
$$\delta(T)\subset B\otimes J+ J\otimes B.$$
Pick an element in $T$. Without loss of generality, we may assume that the element is of the form $y_1y_2$ where $y_i\in B^+$. Moreover, since $B$ is connected as a coalgebra, we can assume that $\Delta(y_i)=1\otimes y_i+y_i\otimes 1+u_i$ where $u_i\in B^+\otimes B^+$. Denote $\tau(u_i)$ by $u_i^{op}$ where $\tau: B\otimes B\rightarrow B\otimes B$ is the flip map. Then 
\begin{equation}
\delta(y_1y_2)= u_1\Delta(y_2)-u_1^{op}\Delta^{op}(y_2),
\end{equation}
which is obviously in $ B\otimes J+ J\otimes B$. This completes the proof.
\end{proof}

We use $\overline{\delta}$ to denote the cobracket on $\mathfrak{m}/\mathfrak{m}^2$ defined in the previous lemma. As we have shown in the proof of Proposition \ref{Lie}, the map $\theta: B\rightarrow \mathcal{H}(B)$ induces a Lie algebra isomorphism from $ \mathfrak{m}/\mathfrak{m}^2$ to $L$, the space of primitive elements of $\mathcal{H}(B)$. We denote this isomorphism by $\overline{\theta}$. Now the cobracket $\overline{\delta}$ passes to $L$ via $\overline{\theta}$. Namely, 
$$\delta':=(\overline{\theta}\otimes \overline{\theta})\delta \overline{\theta}^{-1}$$
is a cobracket on $L$.
\begin{proposition}
Suppose that the base field $k$ has characteristic $0$. Let $B$ be a finitely generated connected Poisson Hopf algebra and let $L$ be the space of primitive elements of $\mathcal{H}(B)$. Then $L$ becomes a Lie bialgebra via $\delta'$ defined above.
\end{proposition}
\begin{proof}
Clearly, $L$ is a Lie algebra via the commutator $[a, b]:=ab-ba$. Now we have to check the condition $(\ref{Liebialgebra})$ holds.
For any $a, b\in L$, choose $x, y\in \mathfrak{m}=B^+$ such that $a=\theta(x)$ and $b=\theta(y)$. It is clear that
$$\delta'([a, b])=(\theta\otimes \theta)\delta(\{x, y\}).$$
By definition, $\delta=\Delta-\Delta^{op}$. 
Hence 
\begin{equation*}
\delta'(a)=\theta(x_1)\otimes \theta(x_2)-\theta(x_2)\otimes \theta(x_1),
\end{equation*}
and 
\begin{equation*}
\delta'(b)=\theta(y_1)\otimes \theta(y_2)-\theta(y_2)\otimes \theta(y_1).
\end{equation*}
Notice that for any $u, v\in B$, $\theta(uv)=\e(u)\theta(v)+\e(v)\theta(u)$. Therefore,
\begin{align*}
&\,\,(\theta\otimes \theta)\delta(\{x, y\})\\
=&\,\,(\theta\otimes \theta)(x_1y_1\otimes \{x_2, y_2\}+\{x_1, y_1\}\otimes x_2y_2-x_2y_2\otimes \{x_1, y_1\}-\{x_2, y_2\}\otimes x_1y_1)\\
=&\,\,\theta(y_1)\otimes [\theta(x), \theta(y_2)]+ \theta(x_1)\otimes [\theta(x_2), \theta(y)]+[\theta(x), \theta(y_1)]\otimes \theta(y_2)
+[\theta(x_1), \theta(y)]\otimes \theta(x_2)\\
&\,\,-\theta(y_2)\otimes [\theta(x), \theta(y_1)]-\theta(x_2)\otimes [\theta(x_1), \theta(y)]-[\theta(x), \theta(y_2)]\otimes \theta(y_1)
-[\theta(x_2), \theta(y)]\otimes \theta(x_1)\\
=&\,\, a\cdot \delta'(b)-b\cdot\delta'(a).
\end{align*}
This completes the proof.
\end{proof}

\begin{example}
{\rm Let $B$ be the Poisson Hopf algebra defined in Example \ref{GK3}. As a Hopf algebra, $\mathcal{H}(B)$ is just $U(L)$, where $L$ is the Lie algebra spanned by $\{y_1, y_2, y_3\}$ with the following Lie bracket
\begin{align*}
[y_1, y_2]=0,\\
[y_3, y_1]= \lambda_1y_1+ \alpha y_2,\\
[y_3, y_2]=\lambda_2y_2,
\end{align*}
where $\alpha=0$ if $\lambda_1\neq \lambda_2$ and $\alpha=0$ or $1$ if $\lambda_1=\lambda_2$. The cobracket $\delta$ is given by 
\begin{align*}
\delta(y_1)&=\delta(y_2)=0\\
\delta(x_3)&=2(y_1\otimes y_2-y_2\otimes y_1).
\end{align*}}
\end{example}

\subsection{Normal basis property}
Let $B$ be a Poisson Hopf algebra. In Corollary \ref{freeness}, it is shown that $B\subset B^e$ is an $\mathcal{H}(B)$-extension. In what follows we will investigate the {\it normal basis property} \cite[Definition 8.2.1]{Mo} of this extension.

Let $I$ be the left ideal of $B^e$ generated by $h_x$ where $x\in B$. As shown in Proposition \ref{decomp},
$B^e\cong B\oplus I$ as left $B$-modules and $I$ is isomorphic to the universal Poisson derivation $\Omega_B$.

Consider the map $p:=(Id\otimes \pi h)\circ \Delta: B\rightarrow B\otimes \mathcal{H}(B)$. A straightforward calculation shows the following.
\begin{lemma}
The map $p: B\rightarrow B\otimes \mathcal{H}(B)$ defined above is a Poisson derivation.
\end{lemma}

Since $I$ can be identified as the universal Poisson derivation, there is a $B^e$-module map $\tau: I \rightarrow B\otimes \mathcal{H}(B)$ such that $p=\tau h$. Let $\iota: B\rightarrow B\otimes \mathcal{H}(B)$ be the $B$-module map sending $b$ to $b\otimes 1$. Now there is a (left) $B$-module map 
\begin{equation}\label{Upsilon}
\Upsilon:=\iota\oplus \tau: B^e\rightarrow B\otimes \mathcal{H}(B),
\end{equation}
where we identify $B^e$ with $B\oplus I$.

Next, we are going to show that $\Upsilon: B^e\rightarrow B\otimes \mathcal{H}(B)$ is also a right $\mathcal{H}(B)$-comodule map. That is to say, the diagram\begin{equation}\label{comodule}
\begin{CD}
B^e@>\Upsilon>> B\otimes \mathcal{H}(B)\\
@V\lambda VV   @V\rho VV \\
B^e\otimes \mathcal{H}(B)@>\Upsilon\otimes Id>> B\otimes \mathcal{H}(B)\otimes \mathcal{H}(B)\\
\end{CD}
\end{equation}
commutes, where $\lambda$ and $\rho$ are comodule structure maps. 

Before we prove that $\Upsilon$ is an $\mathcal{H}(B)$-comodule map, a technical lemma regarding the maps $\lambda$ and $\Upsilon$ is needed. First, we will set up some conventions. As before, we let $\theta=\pi h$. Also recall that, for any $a, b\in B$, $h_a\cdot b$ is defined to be $\{a, b\}$.

For a set $S$, a {\it partition} $\mathcal{P}$ of $S$ is a collection of subsets of $S$ such that $\emptyset\notin \mathcal{P}$, $\bigcup_{L\in \mathcal{P}}L=S$ and $L_1\bigcap L_2=\emptyset$ for any distinct $L_1, L_2\in \mathcal{P}$.
For a finite set $S\subset B$ of $n$ elements, we choose a total order on it, say $a^{(n)}>a^{(n-1)}>...> a^{(1)}$. For a subset $J=\{a^{(\ell_i)}\}_{i=1}^s$ of $S$, let 
\begin{equation}\label{hJ}
h_J=h_{\ell_s}h_{\ell_{s-1}}\cdots h_{\ell_1},
\end{equation}
and
\begin{equation}\label{TJ}
T_{J}:=h_{\ell_s}\cdot(h_{\ell_{s-1}}\cdots(h_{\ell_{2}}\cdot a^{(\ell_1)}_1)\cdots)
\end{equation}
where $\ell_s>\cdots> \ell_1$ and $h_i:=h_{a^{(i)}}$. For a partition $\mathcal{P}=\{J_1, \cdots, J_N\}$ of $S$ with $t(J_1)>\cdots> t(J_N)$ where $t(J_i)$ is the smallest element in $J_i$, set 
\begin{equation}\label{TP}
T_\mathcal{P}=T_{J_1}T_{J_2}\cdots T_{J_N}.
\end{equation}
Also, let
\begin{equation}\label{thetaP}
\theta_{\mathcal{P}}:=\theta(t(J_1)_2)\cdots\theta (t(J_N)_2).
\end{equation}

\begin{remark}\label{convention}
For the purpose of our calculation, we adopt the convention that $\mathcal{P}=\{\emptyset\}$ is the only partition of $S=\emptyset$. Also, we set $h_\emptyset=1$, $T_\mathcal{P}=1$ and $\theta_\mathcal{P}=1$, where $\mathcal{P}=\{\emptyset\}$.

For the symbols $h_J$, $T_{\mathcal{P}}$ and $\theta_{\mathcal{P}}$, one should view them as elements in $B^e$, $B$ and $\mathcal{H}(B)$, respectively. Since $B$ is commutative, the order that we arrange $T_{J_i}$ in the definition of $T_\mathcal{P}$ does not matter. Also notice that if $J=\{a^{(i)}\}$, then 
$$T_J=a^{(i)}_1.$$
\end{remark}

Now we are ready for the lemmas.
\begin{lemma}\label{LU}
Retain the above notation. Then
\begin{equation}\label{lambda}
\lambda(h_S)=\sum_{J\subset S}\sum_{\text{partitions $\mathcal{P}$ of $S\setminus J$}}T_{\mathcal{P}}h_{J}\otimes \theta_{\mathcal{P}},
\end{equation}
where the first summation runs through all subsets $J$ of $S$ and the second summation runs through all partitions of $S\setminus J$. Also,
\begin{equation}\label{Upsilon}
\Upsilon(h_S)=\sum_{\text{partitions $\mathcal{P}$ of $S$}}T_{\mathcal{P}}\otimes \theta_{\mathcal{P}}.
\end{equation}
\end{lemma}
\begin{proof}
Suppose that $S$ has $n$ elements with the ordering $a^{(n)}>a^{(n-1)}>...> a^{(1)}$. Also, as before, use $h_i$ for $h_{a^{(i)}}$. We are going to prove $(\ref{lambda})$ by induction on $n$. If $n=1$, then
$$\lambda(h_S)= a^{(1)}_1\otimes \theta(a^{(1)}_2)+h_1\otimes 1.$$
For the right-hand side of $(\ref{lambda})$, $J$ is either $S$ or $\emptyset$. If $J=S$, then $S\setminus J$ is the empty set and by our convention in Remark $\ref{convention}$, $T_\mathcal{P}=\theta_\mathcal{P}=1$. Hence $T_{\mathcal{P}}h_{J}\otimes \theta_{\mathcal{P}}$ gives $h_1\otimes 1$. If $J=\emptyset$, then $T_{\mathcal{P}}h_{J}\otimes \theta_{\mathcal{P}}$ becomes $a^{(1)}_1\otimes \theta(a^{(1)}_2)$. Consequently, the equation $(\ref{lambda})$ holds when $n=1$.

Now assume that $n\ge 2$ and let $S'$ be the subset $\{a^{(n-1)}, \cdots, a^{(1)}\}$. By induction hypothesis,
\begin{align}\label{lambda1}
\lambda(h_S)=&\,\,(a^{(n)}_1\otimes \theta(a^{(n)}_2)+h_n\otimes 1)h_{S'}\\
\nonumber=&\,\,(a^{(n)}_1\otimes \theta(a^{(n)}_2)+h_n\otimes 1)(\sum_{J'\subset S'}\sum_{\text{partitions $\mathcal{P'}$ of $S'\setminus J'$}}T_{\mathcal{P'}}h_{J'}\otimes \theta_{\mathcal{P'}})\\
\nonumber=&\,\,\sum_{J'\subset S'}\sum_{\text{partitions $\mathcal{P'}$ of $S'\setminus J'$}}(a_1^{(n)}T_{\mathcal{P'}}h_{J'}\otimes \theta(a_2^{(n)})\theta_{\mathcal{P'}}+h_nT_{\mathcal{P'}}h_{J'}\otimes \theta_{\mathcal{P'}})\\
\nonumber=&\,\,\sum_{J'\subset S'}\sum_{\text{partitions $\mathcal{P'}$ of $S'\setminus J'$}}(a_1^{(n)}T_{\mathcal{P'}}h_{J'}\otimes \theta(a_2^{(n)})\theta_{\mathcal{P'}}+T_{\mathcal{P'}}h_nh_{J'}\otimes \theta_{\mathcal{P'}}+(h_n\cdot T_{\mathcal{P'}})h_{J'}\otimes \theta_{\mathcal{P'}})
\end{align}

For a partition $\mathcal{P'}$ of $S'\setminus J'$, since $h_{n}$ acts as a derivation on $T_{\mathcal{P'}}$, we have
\begin{equation}
h_{n}\cdot(T_{\mathcal{P'}})=\sum_{\mathcal{P}}T_{\mathcal{P}}
\end{equation}
where $\mathcal{P}$ runs through the partitions of $S\setminus J'$ obtained by adding $a^{(n)}$ to a set in $\mathcal{P'}$. Notice that $\{a^{(n)}\}\notin \mathcal{P}$ by construction. Conversely, let $\mathcal{P}$ be a partition of $S\setminus J'$ and suppose that $\{a^{(n)}\}\notin \mathcal{P}$. Suppose $\mathcal{P}=\{I_1, I_2, \cdots, I_s\}$ and $a^{(n)}\in I_i$ for some $i$. By assumption the set $I_i\setminus \{a^{(n)}\}$ is not empty. Hence $\mathcal{P'}: =\{I_1,\cdots, I_i\setminus \{a^{(n)}\}, \cdots, I_s\}$ is a partition of $S'\setminus J'$ and $T_{\mathcal{P}}$ would appear as a summand of $h_{n}\cdot(T_{\mathcal{P'}})$. To sum up, we have
\begin{equation}
\sum_{J'\subset S'}\sum_{\text{partitions $\mathcal{P'}$ of $S'\setminus J'$}}(h_n\cdot T_{\mathcal{P'}})h_{J'}\otimes \theta_{\mathcal{P'}}= \sum_{\substack{J\subset S\\ a^{(n)}\notin J}}\sum_{\substack{\text{partitions $\mathcal{P}$} \\ \text{of}\,S\setminus J\\ \{a^{(n)}\}\notin \mathcal{P}}}T_{\mathcal{P}}h_J\otimes \theta_{\mathcal{P}}.
\end{equation}

Combining this with $(\ref{lambda1})$, we get
\begin{align*}
\lambda(h_S)=&\,\, \sum_{\substack{J\subset S\\ a^{(n)}\notin J}}\sum_{\substack{\text{partitions $\mathcal{P}$} \\ \text{of}\,S\setminus J\\ \{a^{(n)}\}\in \mathcal{P}}}T_{\mathcal{P}}h_J\otimes \theta_{\mathcal{P}}\\
&\,\,+\sum_{\substack{J\subset S\\ a^{(n)}\in J}}\sum_{\substack{\text{partitions $\mathcal{P}$} \\ \text{of}\,S\setminus J}}T_{\mathcal{P}}h_J\otimes \theta_{\mathcal{P}}
+\sum_{\substack{J\subset S\\ a^{(n)}\notin J}}\sum_{\substack{\text{partitions $\mathcal{P}$} \\ \text{of}\,S\setminus J\\ \{a^{(n)}\}\notin \mathcal{P}}}T_{\mathcal{P}}h_J\otimes \theta_{\mathcal{P}}\\
=&\,\,\sum_{J\subset S}\sum_{\text{partitions $\mathcal{P}$ of $S\setminus J$}}T_{\mathcal{P}}h_{J}\otimes \theta_{\mathcal{P}}.
\end{align*}
Hence the equation $(\ref{lambda})$ holds. The equation $(\ref{Upsilon})$ can be proved similarly so we omit the details here.
\end{proof}

\begin{proposition}\label{comodulemap}
The map $\Upsilon: B^e\rightarrow B\otimes \mathcal{H}(B)$ is a right $\mathcal{H}(B)$-comodule map.
\end{proposition}
\begin{proof}
Since $B^e$ is the direct sum of $B$ and $I$, where $I$ is the left ideal of $B^e$ generated by $h_x$ where $x\in B$, it suffices to show $\rho\Upsilon= (\Upsilon\otimes Id)\lambda$ on $B$ and $I$. For any $b\in B$,
$$\rho\Upsilon(b)= (\Upsilon\otimes Id)\lambda(b)=b\otimes 1\otimes 1.$$
Therefore, it remains to show that $\rho\Upsilon$ and $(\Upsilon\otimes Id)\lambda$ agree on $I$.

First, we are going to show that $(\Upsilon\otimes Id)\lambda$ is a $B^e$-module map when restricted on $I$. Clearly, $(\Upsilon\otimes Id)\lambda$ is a $B$-module map. Hence, by Remark \ref{generalcase}, we only have to show that
\begin{equation}
(\Upsilon\otimes Id)\lambda(h(a^{(n+1)})\cdots h(a^{(1)}))=h_{a_{n+1}} \cdot(\Upsilon\otimes Id)\lambda(h(a^{(n)})\cdots h(a_1))
\end{equation}
for $n\ge 1$.
Let $S'$ be the set $\{a^{(n+1)}, \cdots, a^{(1)}\}$ with the ordering $a^{(n+1)}>\cdots> a^{(1)}$ and let $S$ be the subset $\{a^{(n)}, \cdots, a^{(1)}\}$. Hence with the notation we introduced before Lemma \ref{LU}, we need to show
\begin{equation}
(\Upsilon\otimes Id)\lambda(h_{S'})=h_{n+1}\cdot (\Upsilon\otimes Id)\lambda(h_S).
\end{equation}
By Lemma \ref{LU}, we have
\begin{align*}
&h_{n+1}\cdot (\Upsilon\otimes Id)\circ\lambda(h_S)\\
=&h_{n+1}\cdot(\Upsilon\otimes Id(\sum_{J\subset S}\sum_{\text{partitions $\mathcal{P}$ of $S\setminus J$}}T_{\mathcal{P}}h_{J}\otimes \theta_{\mathcal{P}}))\\
=& h_{n+1}\cdot(\sum_{J\subset S}\sum_{\text{partitions $\mathcal{P}$ of $S\setminus J$}}\sum_{\text{partitions $\mathcal{A}$ of $J$}}T_{\mathcal{P}}T_{\mathcal{A}}\otimes\theta_{\mathcal{A}}\otimes \theta_{\mathcal{P}})\\
=& \sum_{J\subset S}\sum_{\text{partitions $\mathcal{P}$ of $S\setminus J$}}\sum_{\text{partitions $\mathcal{A}$ of $J$}}(h_{n+1}\cdot(T_{\mathcal{P}}T_{\mathcal{A}})\otimes\theta_{\mathcal{A}}\otimes \theta_{\mathcal{P}}\\
+&a^{(n+1)}_1T_{\mathcal{P}}T_{\mathcal{A}}\otimes\theta(a^{(n+1)}_2)\theta_{\mathcal{A}}\otimes \theta_{\mathcal{P}}
+a^{(n+1)}_1T_{\mathcal{P}}T_{\mathcal{A}}\otimes\theta_{\mathcal{A}}\otimes \theta(a^{(n+1)}_2)\theta_{\mathcal{P}}).
\end{align*}
Hence $h_{n+1}\cdot (\Upsilon\otimes Id)\circ\lambda(h_S)$ can be broken up into the sum of the following expressions.
\begin{equation}\label{sum1}
\sum_{J\subset S}
\sum_{\substack{\text{partitions $\mathcal{P}$} \\ \text{of}\,S\setminus J}}
\sum_{\text{partitions $\mathcal{A}$ of $J$}}
h_{n+1}\cdot(T_{\mathcal{P}})T_{\mathcal{A}}\otimes\theta_{\mathcal{A}}\otimes \theta_{\mathcal{P}},
\end{equation}
\begin{equation}\label{sum2}
\sum_{J\subset S}
\sum_{\substack{\text{partitions $\mathcal{P}$} \\ \text{of}\,S\setminus J}}
\sum_{\text{partitions $\mathcal{A}$ of $J$}}
T_{\mathcal{P}}h_{n+1}\cdot(T_{\mathcal{A}})\otimes\theta_{\mathcal{A}}\otimes \theta_{\mathcal{P}},
\end{equation}
\begin{equation}\label{sum3}
\sum_{J\subset S}
\sum_{\substack{\text{partitions $\mathcal{P}$} \\ \text{of}\,S\setminus J}}
\sum_{\text{partitions $\mathcal{A}$ of $J$}}
a^{(n+1)}_1T_{\mathcal{P}}T_{\mathcal{A}}\otimes\theta(a^{(n+1)}_2)\theta_{\mathcal{A}}\otimes \theta_{\mathcal{P}},
\end{equation}
\begin{equation}\label{sum4}
\sum_{J\subset S}
\sum_{\substack{\text{partitions $\mathcal{P}$} \\ \text{of}\,S\setminus J}}
\sum_{\text{partitions $\mathcal{A}$ of $J$}}
a^{(n+1)}_1T_{\mathcal{P}}T_{\mathcal{A}}\otimes\theta_{\mathcal{A}}\otimes \theta(a^{(n+1)}_2)\theta_{\mathcal{P}}.
\end{equation}
On the other hand,
\begin{equation}\label{bigsum'}
(\Upsilon\otimes Id)\lambda(h_{S'})
 =\sum_{J'\subset S'}
\sum_{\substack{\text{partitions $\mathcal{P}'$} \\ \text{of}\,S'\setminus J'}}
 \sum_{\text{partitions $\mathcal{A}'$ of $J'$}}T_{\mathcal{P'}}T_{\mathcal{A'}}\otimes\theta_{\mathcal{A'}}\otimes \theta_{\mathcal{P'}}.
\end{equation}
Now, the sum (\ref{bigsum'}) can be rewritten into the sum of the following:

\begin{equation}\label{sum1'}
\sum_{\substack{J'\subset S'\\ a^{(n+1)}\notin J'}}\sum_{\substack{\text{partitions $\mathcal{P}'$} \\ \text{of}\,S'\setminus J'\\ \{a^{(n+1)}\}\notin \mathcal{P'}}}\sum_{\substack{\text{partitions $\mathcal{A}'$ of $J'$} }}
T_{\mathcal{P'}}T_{\mathcal{A'}}\otimes\theta_{\mathcal{A'}}\otimes \theta_{\mathcal{P'}}.
\end{equation}
\begin{equation}\label{sum2'}
\sum_{\substack{J'\subset S'\\ a^{(n+1)}\in J'}}\sum_{\substack{\text{partitions $\mathcal{P}'$} \\ \text{of}\,S'\setminus J'}}\sum_{\substack{\text{partitions $\mathcal{A}'$ of $J'$}\\ \{a^{(n+1)}\}\notin \mathcal{A'} }}T_{\mathcal{P'}}T_{\mathcal{A'}}\otimes\theta_{\mathcal{A'}}\otimes \theta_{\mathcal{P'}},
\end{equation}
\begin{equation}\label{sum3'}
\sum_{\substack{J'\subset S'\\ a^{(n+1)}\in J'}}
\sum_{\substack{\text{partitions $\mathcal{P}'$} \\ \text{of}\,S'\setminus J'}}
\sum_{\substack{\text{partitions $\mathcal{A}'$ of $J'$}\\ \{a^{(n+1)}\}\in \mathcal{A'} }}T_{\mathcal{P'}}T_{\mathcal{A'}}\otimes\theta_{\mathcal{A'}}\otimes \theta_{\mathcal{P'}},
\end{equation}

\begin{equation}\label{sum4'}
\sum_{\substack{J'\subset S'\\ a^{(n+1)}\notin J'}}\sum_{\substack{\text{partitions $\mathcal{P}'$} \\ \text{of}\,S'\setminus J'\\ \{a^{(n+1)}\}\in \mathcal{P'}}}\sum_{\substack{\text{partitions $\mathcal{A}'$ of $J'$} }}T_{\mathcal{P'}}T_{\mathcal{A'}}\otimes\theta_{\mathcal{A'}}\otimes \theta_{\mathcal{P'}},
\end{equation}

In (\ref{sum3'}), $T_{\mathcal{A'}}$ is of the form $a^{(n+1)}_1T_{\mathcal{A}}$ and $\theta_{\mathcal{A'}}$ is of the form $\theta(a^{(n+1)}_2)\theta_{\mathcal{A}}$, where $\mathcal{A}=\mathcal{A'}\setminus\{a^{(n+1)}\}$ is a partition of $S$. Hence (\ref{sum3'}) is equal to (\ref{sum3}). Similarly, one can show that (\ref{sum4'}) is equal to (\ref{sum4}).

For a partition $\mathcal{A}$ of a subset $J$ of $S$, since $h_{n+1}$ acts as a derivation on $T_{\mathcal{A}}$, we have
\begin{equation}
h_{n+1}\cdot(T_{\mathcal{A}})=\sum_{\mathcal{A}'}T_{\mathcal{A}'}
\end{equation}
where $\mathcal{A}'$ runs through all partitions of $J'=J\cup \{a^{(n+1)}\}\subset S'$ obtained by adding $a^{(n+1)}$ to a set in $\mathcal{A}$. Notice that $\{a^{(n+1)}\}\notin \mathcal{A}'$ by construction. Conversely, let $\mathcal{A}'$ be a partition of $J'=J\cup \{a^{(n+1)}\}\subset S'$ and suppose that $\{a^{(n+1)}\}\notin \mathcal{A}'$. Suppose $\mathcal{A}'=\{I_1, I_2, \cdots, I_s\}$ and $a^{(n+1)}\in I_i$ for some $i$. By assumption the set $I_i\setminus \{a^{(n+1)}\}$ is not empty. Hence $\mathcal{A}: =\{I_1,\cdots, I_i\setminus \{a^{(n+1)}\}, \cdots, I_s\}$ is a partition of $S$ and $T_{\mathcal{A}'}$ would appear as a summand of $h_{n+1}\cdot(T_{\mathcal{A}})$. Consequently, (\ref{sum2}) is equal to (\ref{sum2'}). Similarly, (\ref{sum1}) is equal to (\ref{sum1'}). This completes the proof.
\end{proof}

\begin{proposition}\label{injectivity}
The map $\Upsilon: B^e\rightarrow B\otimes \mathcal{H}(B)$ is injective.
\end{proposition}
\begin{proof}
If $\Upsilon$ is not injective, by Proposition \ref{comodulemap}, $\ker \Upsilon$ is a non-zero subcomodule of $B^e$. By \cite[5.1.1]{Mo}, $\ker \Upsilon$ must contain a simple subcomodule, i.e. $\ker \Upsilon \cap s(B^e)\neq \{0\}$, where $s(B^e)$ is the socle of the right $\mathcal{H}(B)$-comodule $B^e$. As stated in \cite[Section 2]{AD}, $s(B^e)=\lambda^{-1}(B^e\otimes \mathcal{H}(B)_0)$. By Proposition \ref{primitivegen}, $\mathcal{H}(B)_0$, the coradical of $\mathcal{H}(B)$, is equal to $k1$. Hence, $s(B^e)=(B^e)^{co\mathcal{H}(B)}$. Now it follows from Proposition \ref{freeness} that $s(B^e)=B$. However, by construction, $\Upsilon$ is injective when restricted on $B$, i.e. $\ker \Upsilon \cap B= \{0\}$. This is a contradiction and therefore $\Upsilon$ is an injective map.
\end{proof}

Now it is very tempting to claim that $\Upsilon$ is a bijection. Unfortunately, we are only able to show the surjectivity of $\Upsilon$ when the Poisson Hopf algebra $B$ is pointed.

\begin{theorem}\label{normalbasis}
Let $B$ be a pointed Poisson Hopf algebra. Then the $\mathcal{H}(B)$-extension $B\subset B^e$ has the normal basis property. That is, 
\begin{equation*}
B^e\cong B\otimes \mathcal{H}(B)
\end{equation*}
as left $B$-modules and right $\mathcal{H}(B)$-modules. The isomorphism is given by the map $\Upsilon$ defined in (\ref{Upsilon}).
\end{theorem}
\begin{proof}
Thanks to Proposition \ref{injectivity}, all we have to show is that $\Upsilon$ is surjective. As shown in the proof of Proposition \ref{primitivegen}, $\mathcal{H}(B)$ is generated as an algebra by elements $\theta(a)$ where $a\in B$. Since $\Upsilon$ is a $B$-module map, it suffices to show that for any positive integer $\ell$ and $a_1, \cdots, a_\ell\in B$, the element
\begin{equation}\label{generic}
1\otimes \theta(a_\ell)\cdots \theta(a_1)
\end{equation}
is in the image of $\Upsilon$. In fact, we will show the element $(\ref{generic})$ is in $\Upsilon(I)$, where $I$ is the left ideal of $B^e$ generated by elements $h_b$ where $b\in B$. We proceed by induction on $\ell$.

Suppose that $\ell=1$. Recall that $\{B_i\}_{i\ge 0}$ is the coradical filtration of $B$. Then $a_1\in B_n$ for some $n$. If $n=0$, without loss of generality, we may assume $a_1=g$ where $g$ is a group-like element. Hence by the definition of $\Upsilon$,  one has $1\otimes \theta(a_1)=\Upsilon(g^{-1}h_{a_1})$. Now assume that $n\ge 1$. By \cite[Theorem 5.4.1]{Mo}, we may assume $a_1$ is such that 
\begin{equation*}
\Delta(a_1)=g\otimes a_1+ a_1\otimes x+ \sum_{i}b_i\otimes c_i, 
\end{equation*}
where $g, x$ are group-like elements and $b_i, c_i\in B_{n-1}$. By the induction on $n$, there exists $z\in I$ such that $\Upsilon(z)=a_1\otimes \theta(x)+\sum_ib_i\otimes \theta(c_i)$. Therefore a direct calculation shows that
\begin{equation*}
1\otimes \theta(a_1)= \Upsilon(g^{-1}h_{a_1}-g^{-1}z).
\end{equation*}
Clearly, the element $g^{-1}h_{a_1}-g^{-1}z\in I$.

Now suppose that $\ell\ge 2$. By the induction hypothesis on $\ell$, there exists $y\in I$ such that $\Upsilon(y)=1\otimes \theta(a_{\ell-1})\cdots \theta(a_1)$. Again, $a_\ell\in B_n$ for some $n$. If $n=0$, without loss of generality, we may assume $a_\ell=g$ where $g$ is a group-like element. By construction, $\Upsilon$ is a $B^e$-module map on $I$. Hence
\begin{equation*}
\Upsilon(h_{a_\ell}y)= h_{a_\ell}\cdot\Upsilon(y)=g\otimes  \theta(a_\ell)\cdots \theta(a_1).
\end{equation*}
Therefore,
\begin{equation*}
1\otimes  \theta(a_\ell)\cdots \theta(a_1)=\Upsilon (g^{-1}h_gy).
\end{equation*}
Now assume that $n\ge 1$. Without loss of generality, we may assume $a_\ell$ is such that 
\begin{equation*}
\Delta(a_\ell)=g\otimes a_\ell+ a_\ell\otimes x+ \sum_{i}b_i\otimes c_i, 
\end{equation*}
where $g, x$ are group-like elements and $b_i, c_i\in B_{n-1}$. Now
\begin{equation*}
\Upsilon(h_{a_\ell}y)= h_{a_\ell}\cdot\Upsilon(y)=g\otimes  \theta(a_\ell)\cdots \theta(a_1)+ a_\ell\otimes \theta(x) \theta(a_{\ell-1})\cdots \theta(a_1)+ \sum_i b_i\otimes \theta(c_i)  \theta(a_{\ell-1})\cdots \theta(a_1).
\end{equation*}
 By the induction on $n$, there exists $w\in I$ such that $\Upsilon(w)=a_\ell\otimes \theta(x) \theta(a_{\ell-1})\cdots \theta(a_1)+ \sum_i b_i\otimes \theta(c_i)  \theta(a_{\ell-1})\cdots \theta(a_1)$. Therefore,
 \begin{equation*}
 1\otimes  \theta(a_\ell)\cdots \theta(a_1)=\Upsilon (g^{-1}h_{a_\ell}y-g^{-1}w).
 \end{equation*}
 Obviously the element $g^{-1}h_{a_\ell}y-g^{-1}w\in I$. This completes the proof.
\end{proof}

\subsection{Galois extension} In the previous section, we have shown that for a pointed Poisson Hopf algebra $B$, the $\mathcal{H}(B)$-extension $B\subset B^e$ has the normal basis property. In fact, the extension enjoys another nice property - it is a (right) $\mathcal{H}(B)$-Galois extension. For the definition of a Hopf-Galois extension, one can refer to \cite[Definition 8.1.1]{Mo}. It amounts to show, in our setting, that the map $\beta: B^e\otimes_BB^e\rightarrow B^e\otimes \mathcal{H}(B)$, given by $x\otimes y\mapsto (x\otimes 1)\lambda(y)$, is bijective.

\begin{proposition}
Let $B$ be a pointed Poisson Hopf algebra. Then the $\mathcal{H}(B)$-extension $B\subset B^e$ is Galois.
\end{proposition}
\begin{proof} By Proposition \ref{freeness}, $B^e$ is an injective comodule over $\mathcal{H}(B)$. Hence, by \cite[Corollary 2.4.9]{Sch}, we only have to show the Galois map $\beta$ is surjective. Clearly, $\beta$ is a left $B^e$-module map, where the $B^e$-module structure on $B^e\otimes_BB^e$ is given by $z(x\otimes y):= zx\otimes y$.
Therefore, it suffices to show that for any positive integer $\ell$ and $a_1, \cdots, a_\ell\in B$, the element
\begin{equation*}
1\otimes \theta(a_1)\cdots \theta(a_\ell)
\end{equation*}
is in the image of $\beta$. The proof is very similar to that of Theorem \ref{normalbasis}. We proceed by induction on $\ell$.

Suppose that $\ell=1$. Recall that $\{B_i\}_{i\ge 0}$ is the coradical filtration of $B$. Then $a_1\in B_n$ for some $n$. If $n=0$, without loss of generality, we may assume $a_1=g$ where $g$ is a group-like element.  
A direct calculation shows that
\begin{equation*}
1\otimes \theta(a_1)= \beta(g^{-1}\otimes h_g-g^{-1}h_g\otimes 1).
\end{equation*}
Now assume that $n\ge 1$. By \cite[Theorem 5.4.1]{Mo}, we may assume $a_1$ is such that 
\begin{equation*}
\Delta(a_1)=g\otimes a_1+ a_1\otimes x+ \sum_{i}b_i\otimes c_i, 
\end{equation*}
where $g, x$ are group-like elements and $b_i, c_i\in B_{n-1}$. 
So we have
$$\lambda(h_{a_1})=g\otimes \theta(a_1)+\e(a_1)h_g\otimes 1+\sum_i(b_i\otimes \theta(c_i)+\e(c_i)h_{b_i}\otimes 1):=g\otimes \theta(a_1)+ w.$$
By the induction on $n$, there exists $z\in B^e\otimes_BB^e$ such that $\beta(z)=w$. Therefore,
\begin{equation*}
1\otimes \theta(a_1)= \beta(g^{-1}h-g^{-1}z).
\end{equation*}
Clearly, the element $g^{-1}h_{a_1}-g^{-1}z\in I$.

Now suppose that $\ell\ge 2$. By the induction hypothesis on $\ell$, there exists $y\in B^e\otimes_BB^e$ such that $\beta(y)=1\otimes \theta(a_{1})\cdots \theta(a_{\ell-1})$. Again, $a_\ell\in B_n$ for some $n$. If $n=0$, without loss of generality, we may assume $a_\ell=g$ where $g$ is a group-like element. 
Now
\begin{equation*}
\beta(yh_{a_\ell})=\beta(y)\lambda(h_{a_\ell})=g\otimes  \theta(a_1)\cdots \theta(a_\ell)+ h_{a_\ell}\otimes \theta(a_1)\cdots \theta(a_{\ell-1}).
\end{equation*}
Therefore,
\begin{equation*}
1\otimes  \theta(a_1)\cdots \theta(a_\ell)=\beta (g^{-1}yh_g-g^{-1}h_gy).
\end{equation*}
Now assume that $n\ge 1$. Without loss of generality, we may assume that $a_\ell$ is such that 
\begin{equation*}
\Delta(a_\ell)=g\otimes a_\ell+ a_\ell\otimes x+ \sum_{i}b_i\otimes c_i, 
\end{equation*}
where $g, x$ are group-like elements and $b_i, c_i\in B_{n-1}$. Now
\begin{align*}
\beta(yh_{a_\ell})=&\beta(y)\lambda(h_{a_\ell})\\
=&g\otimes\theta(a_1)\cdots \theta(a_\ell)+  \e(a_\ell)h_g\otimes \theta(a_{1})\cdots \theta(a_{\ell-1}) \\
&+\sum_i(b_i\otimes \theta(a_{1})\cdots \theta(a_{\ell-1})\theta(c_i)+\e(c_i)h_{b_i}\otimes \theta(a_{1})\cdots \theta(a_{\ell-1}))\\
:= &g\otimes\theta(a_1)\cdots \theta(a_\ell)+w'.
\end{align*}
 By the induction on $n$, there exists $z'\in  B^e\otimes_BB^e$ such that $\beta(z')=w'$. Therefore,
 \begin{equation*}
 1\otimes  \theta(a_\ell)\cdots \theta(a_1)=\beta(g^{-1}yh_{a_\ell}-g^{-1}z').
 \end{equation*}
 This completes the proof.
\end{proof}

By the previous results, the structure of $B^e$, where $B$ is a pointed Poisson Hopf algebra, is rather clear. In fact, since the $\mathcal{H}(B)$-extension $B\subset B^e$ is Galois and has the normal basis property, it follows from \cite[Theorem 8.2.4]{Mo} that the extension is $\mathcal{H}(B)$-cleft in the sense of \cite[Definition 7.2.1(2)]{Mo}. Now, by \cite[Theorem 7.2.2]{Mo}, the
algebra $B^e$ can be recovered as the {\it crossed product} of $B$ with $\mathcal{H}(B)$.
\begin{proposition}\label{crossproduct}
Let $B$ be a pointed Poisson Hopf algebra. Then 
$$B^e\cong B\#_\sigma \mathcal{H}(B).$$
\end{proposition}

\medskip
\section{The structure of $(\gr B)^e$}\label{graded}

In the previous section, we showed that for a pointed Poisson Hopf algebra $B$, $B^e$ is isomorphic to $B\#_\sigma \mathcal{H}(B)$. However, the existence of the cocycle $\sigma$ still makes the structure of $B\#_\sigma \mathcal{H}(B)$ very complicated. 

In this section, we focus on a particular type of pointed Poisson Hopf algebras $B$ such that $\{g, t\}=0$ for any $g, t\in G(B)$. We will show that $C=\gr B$ carries an induced Poisson structure and $C^e$ has a nice decomposition (Proposition \ref{Decomp}).
\begin{lemma}
Let $B$ be a Pointed Poisson Hopf algebra and denote its coradical filtration by $\{B_n\}_{n\ge 0}$. If $\{g, t\}=0$ for any $g, t\in G(B)$, then 
$$\{B_i, B_j\}\subset B_{i+j},$$
for any $i, j\ge 0$.
\end{lemma}
\begin{proof}
Easy by induction.
\end{proof}

In the previous lemma, the assumption that $\{g ,t\}=0$ for any $g, t\in G(B)$ is essential. As shown in Example (\ref{group}), it might be the case that $\{g, t\}\neq 0$ for some $g, t\in G(B)$.

\begin{proposition}\label{gradedPoisson}
Let $B$ be a Pointed Poisson Hopf algebra such that $\{g, t\}=0$ for any $g, t\in G(B)$. Then $C=\gr B$ inherits a Poisson Hopf algebra structure from $B$ with $C(n)=B_n/B_{n-1}$. Moreover, $C$ is a graded Poisson Hopf algebra in the sense that it is a graded Hopf algebra and 
$$\{C(i), C(j)\}\subset C(i+j).$$
\end{proposition}
\begin{proof}
It is well known that $C$ is a graded Hopf algebra. For any $x\in B_i$ and $y\in B_j$, let $\overline{x}$ and $\overline{y}$ be the corresponding elements in $C(i)$ and $C(j)$, respectively. Then the Poisson bracket on $C$ is defined by 
$$\{\overline{x}, \overline{y}\}=\overline{\{x, y\}},$$
where $\overline{\{x, y\}}$ is the element in $C(i+j)$ represented by $\{x, y\}\in B_{i+j}$. It is clear that $\{C(i), C(j)\}\subset C(i+j)$. It is routine to check that $C$ is a Poisson Hopf algebra with the Poisson bracket defined above.
\end{proof}

\subsection{The structure of $C^e$} Throughout this subsection, let $B$ be a Pointed Poisson Hopf algebra such that $\{g, t\}=0$ for any $g, t\in G(B)$. We simply denote $G(B)$ by $G$ and let $C=\gr B$. There is a canonical projection $\pi : C\rightarrow kG$ such that $\pi\iota=id$ where $\iota$ is the inclusion $kG\rightarrow C$. So by a well-know result of Radford \cite{Ra3},
\begin{equation}\label{bi}
C\cong R\#kG
\end{equation}
as Hopf algebras, where $R=C^{co \pi}=\{x\in C: (id\otimes\pi)\Delta(x)=x\otimes 1\}$. 

The algebra $R\#kG$ is called the {\it biproduct} of $R$ and $kG$.
In fact, every element in $C$ can be written as a linear combination of elements of the form $xg$ where $x\in R$ and $g\in G$ and the isomorphism sends $xg$ to $x\#g$. Moreover, $R$ is a braided graded Hopf algebra in $^G_G\mathcal{M}$, the Yetter-Drinfeld category over $G$. The left $G$-action on $R$ is given by the conjugation (which is trivial in this case since $C$ is commutative) and the left $G$-coaction is given by $(\pi\otimes id)\Delta$. 

\begin{lemma}\label{Poissonmap}
Retain the above notation. Then the map $\iota$ and $\pi$ are Poisson maps.
\end{lemma}
\begin{proof}
The map $\iota$ is clearly a Poisson map. Notice that by definition, $\ker \pi=\bigoplus_{i=1}^{\infty}C(i)$. Now it follows from Proposition \ref{gradedPoisson} that $\ker \pi$ is a Poisson ideal. Hence $\pi$ is also a Poisson map.
\end{proof}

Now we make an easy observation before moving on to the next lemma. Since $\ker \pi$ is a Poisson ideal of $C$, for any $u\in \ker \pi\otimes \ker \pi$ and $w\in C\otimes C$, by the definition of the Poisson bracket on $C\otimes C$, we have 
\begin{equation}\label{fact}
\{u, w\}\in \ker \pi\otimes \ker \pi.
\end{equation}

\begin{lemma}\label{inR}
Retain the above notation. Then the following are true.
\begin{enumerate}
\item[\textup{(1)}] The algebra $R$ is a Poisson subalgebra of $C$, 

\item[\textup{(2)}] $\{g, \{y, g^{-1}\}\}\in R$ and $g^{-1}\{g, y\}\in R$ for any $y\in R$ and $g\in G$.
\end{enumerate}
\end{lemma}
\begin{proof}
For any $a, b\in R$, we have to show that $(id\otimes\pi)\Delta(\{a, b\})=\{a, b\}\otimes 1$. Without loss of generality, we may assume that $a, b$ are homogeneous of degree $\ge1$ and therefore $a, b\in \ker \pi$. Since $C$ is a pointed and graded coalgebra, we may further assume that $\Delta(a)$ and $\Delta(b)$ are of the forms
$\Delta(a)=a\otimes 1+g\otimes a+u$,
and 
$\Delta(b)=b\otimes 1+t\otimes b+v$, where $g, t$ are group-like elements and $u, v\in \ker \pi \otimes \ker \pi$.
Then a direct calculation shows that
\begin{align*}
\Delta(\{a, b\})=&\,\,\{\Delta(a), \Delta(b)\}\\
=&\,\,\{a, b\}\otimes 1+\{a, h\}\otimes b+ \{g, b\}\otimes a+gh\otimes \{a, b\}+\{u, \Delta(a)\}+\{a\otimes1+g\otimes a, v\}.
\end{align*}
By using $(\ref{fact})$ and the fact that $a, b$ are in the Poisson ideal $\ker\pi$, one sees that 
$$(id\otimes\pi)\Delta(\{a, b\})=\{a, b\}\otimes 1.$$
Part $(2)$ can be proved similarly.
\end{proof}

For any $g\in G$ and $y\in R$, we let $g*y=g^{-1}\{g,y\}$. Then the following lemma is clear. Recall that $\{g, t\}=0$ for any $g, t\in G$.
\begin{lemma}\label{star}
Retain the above notation. For any $g, t\in G$ and $y\in R$, the following statements are true.
\begin{enumerate}
\item[\textup{(1)}] $(gt)*y=g*y+t*y$,
\item[\textup{(2)}] $g*(t*y)= g^{-1}t^{-1}\{g, \{t, y\}\}$,
\item[\textup{(3)}] $g*(g*y)= \{g, \{y, g^{-1}\}\}$.
\end{enumerate}
\end{lemma}
\begin{proof}
Part (1) and (2) can be easily verified by calculation. Notice that $(1)$ indicates that $g^{-1}*y=-g*y$. Hence
$$g*(g*y)=-g*(g^{-1}*y)=-\{g, \{g^{-1}, y\}\}=\{g, \{y, g^{-1}\}\}.$$
This completes the proof.
\end{proof}

By the functoriality of taking the enveloping algebra, the maps $\iota$ and $\pi$ induce maps $\iota^e: (kG)^e\rightarrow C^e$ and $\pi^e: C^e\rightarrow (kG)^e$ such that $\pi^e\iota^e=id$ (which implies that $\iota^e$ is injective). Consequently,
\begin{equation}
C^e\cong T\#(kG)^e,
\end{equation}
where $T$ is the right coinvariants of the map $\pi^e$. The inclusion $R\rightarrow C$ induces an algebra map $\Psi: R^e\rightarrow C^e$, which satisfies the following commutative diagram
\[
\begin{CD}
R@>\alpha>> R^e@<\beta<< R\\
@VVV   @V\Psi VV   @VVV\\
C@>m>> C^e@<h<< C\,,\\
\end{CD}
\]
where $\alpha$ and $\beta$ are structure maps of $R^e$.

\begin{proposition} Retain the above notation. Then $\Psi(R^e)\subset T$. Consequently, we can view $\Psi$ as a map from $R^e$ to $T$. 
\end{proposition}
\begin{proof}
Clearly, as an algebra, $\Psi(R^e)$ is generated by $\Psi(\alpha_x)=m_x$ and $\Psi(\beta_x)=h_x$ where $x\in R$.
Without loss of generality, we may assume that $x$ is homogeneous of degree $\ge 1$ and thus $x\in \ker \pi$.
Following our convention, we simply use $x$ for $m_x$. By definition,
$$T=\{y\in C^e: (id\otimes\pi^e)\Delta(y)=y\otimes 1\}.$$
By the definition of $\pi^e$, we have the following commutative diagram
\begin{equation}\label{commute}
\begin{CD}
C@>m>>C^e@<h<< C\\
@V\pi VV   @V\pi^e VV   @V\pi VV\\
kG@>m>> (kG)^e@<h<< kG\,,\\
\end{CD}
\end{equation}
where by a little abuse of notation, we still use $m, h$ for the structure maps of $(kG)^e$.
The left square of $(\ref{commute})$ says that $\pi^e$, when restricted to $C$, is equal to $\pi$. Also, we have seen that $C$ is a Hopf subalgebra of $C^e$. Hence
$$(id\otimes\pi^e)\Delta(x)=(id\otimes\pi)\Delta(x)=x\otimes 1.$$
On the other hand, the right square of $(\ref{commute})$ indicates that $\pi^e(h_a)=h\pi(a)=0$ for any $a\in \ker\pi$. As in the proof of Lemma $\ref{inR}$, we may assume that $\Delta(x)=x\otimes 1+g\otimes x+\sum_is_i\otimes t_i$, where $g$ is group-like and $s_i, t_i\in \ker \pi\otimes \ker \pi$.
Therefore,
$$(id\otimes\pi^e)\Delta(h_x)=(id\otimes\pi^e)(h_x\otimes 1+g\otimes h_x+h_g\otimes x+\sum_i(s_i\otimes h_{t_i}+h_{s_i}\otimes t_i))= h_x\otimes 1.$$
This completes the proof.
\end{proof}

\begin{proposition}
The algebra $R^e$ becomes a $(kG)^e$-module algebra via 
\begin{align*}
g\cdot y= y,\\
g\cdot \beta_y=\beta_y+g^{-1}\{g, y\},\\
h_g\cdot y=g^{-1}\{g, y\},\\
h_g\cdot \beta_y=\beta_{g^{-1}\{g, y\}}+\{g, \{y, g^{-1}\}\},
\end{align*}
for any $g\in G$ and $y\in R$. Moreover, the map $\Psi: R^e\rightarrow T$ is a $(kG)^e$-module map.
\end{proposition}
\begin{proof}
By Lemma \ref{inR}, those defining equations make sense. To prove that $R^e$ is a $(kG)^e$-module algebra, it suffices to check that the relations $(\ref{5.1})$ - $(\ref{5.5})$ of $(kG)^e$ are preserved by the generators of $R^e$, and that the relations $(\ref{5.1})$ - $(\ref{5.5})$ of $R^e$ are preserved by the generators of $(kG)^e$ as well. For the sake of brevity, we omit the details.

To show that $\Psi: R^e\rightarrow T$ is a $(kG)^e$-module map, we have to explicitly find out the $(kG)^e$-action on $\Psi(R^e)$. As an algebra, $\Psi(R^e)$ is generated by $\Psi(\alpha_y)=m_y=y$ and $\Psi(\beta_y)=h_y$ where $y\in R$. Hence we only have to specify the $(kG)^e$-action on $y$ and $h_y$. Recall that the $(kG)^e$ action on $T$ is given by
$$a\cdot z=a_1zS(a_2),$$
for any $a\in (kG)^e$ and $z\in T$.
A direct calculation shows that
\begin{align*}
g\cdot y= y,\\
g\cdot h_y=h_y+g\{y, g^{-1}\},\\
h_g\cdot y=g^{-1}\{g, y\},\\
h_g\cdot h_y=h_{g^{-1}\{g, y\}}+\{g, \{y, g^{-1}\}\},
\end{align*}
for any $g\in G$ and $y\in R$. By comparing these with the $(kG)^e$-action on $R^e$ and noticing the fact that both $R^e$ and $T$ are $(kG)^e$-module algebras, we see that $\Psi$ is a $kG$-module map.
\end{proof}

By the previous proposition, $R^e$ is a $(kG)^e$-Hopf module algebra and thus one can form the algebra $R^e\#(kG)^e$. Also, the map $\Psi$ induces an algebra map $\Phi=\Psi\otimes Id: R^e\#(kG)^e\rightarrow T\#(kG)^e\cong C^e$.

There is a canonical algebra map $\gamma: R\#kG\rightarrow R^e\#(kG)^e$ induced by the inclusions $R\rightarrow R^e$ and $kG\rightarrow (kG)^e$. Now we define a linear map $\lambda: C=R\#kG\rightarrow R^e\#(kG)^e$ which sends $x\#g=xg$ to 
$$x\#h_g+\beta_x\#g+g\{x, g^{-1}\}\#g=(x\#1)(1\#h_g)+(1\#g)(\beta_x\#1)=xh_g+g\beta_x.$$
Note that the element $a\in R^e$ (resp. $b\in (kG)^e$) is identified with $a\#1\in R^e\#(kG)^e$ (resp. $1\#b\in R^e\#(kG)^e$).

\begin{proposition}\label{generation}
 Retain the above notation. Then the following are true.
\begin{enumerate}
\item[\textup{(1)}] The map $\lambda$ is a Lie map;
\item[\textup{(2)}] $\gamma(\{a, b\})=[\lambda(a), \gamma(b)]$ and $\lambda(ab)=\gamma(a)\lambda(b)+\gamma(b)\lambda(a)$, for any $a, b\in R\#kG$;
\item[\textup{(3)}] The algebra $R^e\#(kG)^e$ is generated by the image of $\gamma$ and $\lambda$.
\end{enumerate}
\end{proposition}
\begin{proof}
In the following calculation, we denote $g^{-1}\{g, x\}=-g\{g^{-1}, x\}$ by $g*x$ for any $x\in R$ and $g\in G$. 
To show that $\lambda$ is a Lie map, it suffices to show that 
\begin{equation*}
\lambda(\{ xg,yt \})=[\lambda(xg), \lambda(yt)],
\end{equation*}
for any $x, y\in R$ and $g, t\in G$. 

Notice that we have the following formulae,
\begin{align*}
g\beta_x=&\,\,\beta_xg+(g*x)g\\
[h_g,\beta_y]=&\,\,(g*y)h_g+\beta_{g*y}g+ \{g, \{y, g^{-1}\}\}g.
\end{align*}

Now, 
\begin{align*}
\lambda(\{ xg,yt \})=&\,\,\lambda(\{x, y\}gt - (t*x)ygt+x(g*y)gt)\\
=&\,\,\{x, y\}h_{gt} +gt\beta_{\{x, y\}}- (t*x)yh_{gt}-gt\beta_{(t*x)y}\\
 &\,\,+ x(g*y)h_{gt}+gt\beta_{x(g*y)}.
\end{align*}

On the other hand,
\begin{align*}
[\lambda(xg), \lambda(yt)]=&\,\, [xh_g+g\beta_x, yh_t+t\beta_y]\\
=&\,\, x[h_g, y]h_t+y[x, h_t]h_g +t[x, \beta_y]h_g+tx[h_g, \beta_y]+x[h_g, t]\beta_y\\
&+g[\beta_x, y]h_t+yg[\beta_x, h_t]+y[g, h_t]\beta_x+t[g, \beta_y]\beta_x+tg[\beta_x, \beta_y]+g[\beta_x, t]\beta_y.
\end{align*}
By comparing the two equations, it amounts to show the following equality.
\begin{align}\label{simplify}
  &-(t*x)yh_{gt}-gt\beta_{(t*x)y}+ x(g*y)h_{gt}+gt\beta_{x(g*y)}\\
\nonumber=& x[h_g, y]h_t+y[x, h_t]h_g+tx[h_g, \beta_y]+yg[\beta_x, h_t]+t[g, \beta_y]\beta_x+g[\beta_x, t]\beta_y.
\end{align}

The left-hand side of (\ref{simplify}) becomes
\begin{align*}
&-(t*x)ygh_{t}-(t*x)yth_{g}-gty\beta_{t*x}-gt(t*x)\beta_{y}\\
&+x(g*y)gh_{t}+x(g*y)th_{g}+gtx\beta_{g*y}+gt(g*y)\beta_{x}\\
=&-(t*x)ygh_{t}-(t*x)yth_{g}+x(g*y)gh_{t}+x(g*y)th_{g}\\
&-y\beta_{t*x}gt-y((gt)*(t*x))gt-(t*x)\beta_ygt-(t*x)((gt)*y)gt\\
&+x\beta_{g*y}gt+x((gt)*(g*y))gt+(g*y)\beta_xgt+(g*y)((gt)*x)gt.
\end{align*}

Notice that $[h_g, y]=\{g, y\}=(g*y)g$. Hence the right-hand side of (\ref{simplify}) becomes
\begin{align*}
&x(g*y)gh_t- y(t*x)th_g\\
&+tx(g*y)h_g+x\beta_{g*y}tg+ x(t*(g*y))tg+ tx\{g, \{y, g^{-1}\}\}g\\
&-gy(t*x)h_t-y\beta_{t*x}tg- y(g*(t*x))tg- gy\{t, \{x, t^{-1}\}\}t\\
&+(g*y)\beta_xgt+(g*y)(t*x)gt+(g*y)(g*x)gt\\
&-(t*x)\beta_ygt-(t*x)(g*y)gt-(t*x)(t*y)gt.
\end{align*}

Consequently, it amounts to show that 
\begin{align*}
&-y((gt)*(t*x))+x((gt)*(g*y))\\
=&\,\,x(t*(g*y))tg- y(g*(t*x))tg+x\{g, \{y, g^{-1}\}\}- y\{t, \{x, t^{-1}\}\}
\end{align*}
which is true by Lemma \ref{star}. Therefore $\lambda$ is a Lie map. Part (2) can be proved similarly and we leave it to the reader. 

As an algebra, $R^e\#(kG)^e$ is generated by elements of the forms
$$x\#1, \beta_y\#1, 1\# g, 1\# h_t,$$
where $x, y\in R$ and $g, t\in G$. Obviously, $x\#1$ and $1\# g$ are in the image of $\gamma$. Also, $\beta_y\#1=\lambda(y\#1)$ and $1\#h_t=\lambda(1\#t)$. This completes the proof.
\end{proof}

\begin{proposition}\label{mapPhi}
Retain the above notation. Then the diagram
\[
\begin{CD}
R\#kG@>\gamma>> R^e\#(kG)^e@<\lambda<< R\#kG\\
@V=VV   @V\Phi VV   @VV=V\\
R\#kG@>m>> T\#(kG)^e@<h<< R\#kG\\
\end{CD}
\]
commutes. Moreover, the map $\Phi$ is the unique map that makes this diagram commutes. 
\end{proposition}
\begin{proof}
The commutativity of the diagram is easy to check by the definitions of $\gamma$ and $\lambda$. The uniqueness of the map follows from Proposition \ref{generation} (3).
\end{proof}

Putting things together, we have the following theorem.
\begin{theorem}\label{Decomp}
Let $B$ be a pointed Poisson Hopf algebra with group-like elements $G$ such that $\{g, t\}=0$ for any $g, t\in G$. Then the following statements are true.
\begin{enumerate}
\item[\textup{(1)}] $C=\gr B$ is a graded Poisson Hopf algebra. 
\item[\textup{(2)}] $R^e\#(kG)^e\cong C^e$, where $C\cong R\#kG$ is the biproduct decomposition of $C$. The isomorphism is given by the map $\Phi$ as in Proposition \ref{mapPhi}.

\end{enumerate}
\end{theorem}
\begin{proof}
Part (1) is just Proposition \ref{gradedPoisson}.
Part (2) follows from Proposition \ref{mapPhi} and the universal property of $C^e\cong T\#(kG)^e$.
\end{proof}

For the completion of our discussion, we study the structures of $(kG)^e$ and $R^e$ for the rest of the section. 

\subsection{The structure of $(kG)^e$} Let $G$ be a free abelian group of rank $\ell$ with generators $g_1, ...., g_\ell$. Then $kG$ becomes a Poisson Hopf algebra with $\{a, b\}=0$ for any $a, b\in G$. By definition, $(kG)^e$ is generated by $g_1, ...., g_\ell$ and $h_{g_1}, ..., h_{g_\ell}$. 

Let $\mathfrak{g}$ be the abelian Lie algebra of dimension $\ell$ and let $G$ acts trivially on $\mathfrak{g}$. Then we can form the Hopf algebra $S(\mathfrak{g})\otimes kG$.

\begin{proposition}\label{kGE}
Let $G$ be a free abelian group of rank $\ell$ and equip $kG$ with the trivial Poisson bracket. Then 
$$(kG)^e\cong S(\mathfrak{g})\otimes kG$$
as Hopf algebras,
where $\mathfrak{g}$ is the abelian Lie algebra of dimension $\ell$ with trivial $G$-action.
\end{proposition}
\begin{proof}
Let $\alpha$ be the algebra map $kG\rightarrow S(\mathfrak{g})\otimes kG$ sending $g$ to $1\otimes g$. Let $\{g_1, ...., g_\ell\}$ be a set of free generators of $G$ and $\{y_1, ...., y_\ell\}$ be a basis of $\mathfrak{g}$. By the defining relations of $(kG)^e$, there is a unique algebra map $\Xi: S(\mathfrak{g})\otimes kG$ to $(kG)^e$ sending $1\otimes g$ to $g$ and $y_i\otimes1$ to $h_{g_i}$. The map $\Xi$ is actually a Hopf algebra map.

Define a linear map $\beta: kG\rightarrow S(\mathfrak{g})\otimes kG$ sending $g_1^{i_1}g_2^{i_2}\cdots g_\ell^{i_\ell}$ to 
\begin{equation}
\sum_{j=1}^\ell i_jy_j\otimes g_1^{i_1}\cdots g_j^{i_j-1}\cdots 
\end{equation}

Clearly, $\beta$ is a Lie algebra map. Moreover, we have $\alpha(\{a, b\})=[\beta(a), \alpha(b)]$ and $\beta(ab)=\alpha(a)\beta(b)+\alpha(b)\beta(a)$ for any $a, b\in kG$.
Now $\Xi$ is the unique map that makes the following diagram commute
\[
\begin{CD}
kG@>\alpha>> S(\mathfrak{g})\otimes kG@<\beta<< kG\\
@V=VV   @V\Xi VV   @VV=V\\
kG@>m>> (kG)^e@<h<< kG\,.\\
\end{CD}
\]
Therefore, by the universal property of $(kG)^e$, the map $\Xi$ is an isomorphism.
\end{proof}

\subsection{The structure of $R^e$} The algebra $R$ is a braided graded Hopf algebra in the Yetter-Drinfeld category $\mathcal{M}^G_G$. Since the Hopf algebra $C=\gr B$ is commutative, the braiding on $R\otimes R$ is trivial, i.e. the braiding sends $x\otimes y$ to $y\otimes x$. Therefore $R$ is actually a Hopf algebra in the usual sense. Also, as a coalgebra, $R$ is connected and coradically graded. Putting all these together, we have
\begin{proposition}\label{Rpolynomial}
 Retain the above notation and assume that the base field $k$ has characteristic $0$. If $R$ is finitely generated, then $R$ is isomorphic to a polynomial algebra of finitely many variables. Consequently, $R^e$ has a basis of the form
 $$x_1^{j_1}\cdots x^{j_m}_m \beta^{t_1}_{x_1}\cdots \beta^{t_m}_{x_m},$$
 where $\{x_1, \cdots, x_m\}\subset R^+$ is a minimal generating set of $R$.
\end{proposition}
\begin{proof}
As just discussed, $R$ is a connected Hopf algebra and is coradically graded. Hence $R$ is a polynomial algebra by \cite[Theorem 1.13]{I}. The second part follows from Proposition \ref{PBW} or \cite[Theorem 3.7]{OPS}.
\end{proof}

\end{document}